\documentclass{amsart}

\usepackage{amssymb}
\usepackage{amsmath}
\usepackage{url}
\usepackage{enumerate}
\usepackage{graphicx}

\usepackage{extarrows}

\newcommand{\erase}[1]{}

\newtheorem{theorem}{Theorem}[section]
\newtheorem{lemma}[theorem]{Lemma}
\newtheorem{proposition}[theorem]{Proposition}
\newtheorem{corollary}[theorem]{Corollary}

\newtheorem{_definition}[theorem]{Definition}
\newenvironment{definition}{\begin{_definition}\rm}{\end{_definition}}

\newtheorem{_remark}[theorem]{\it Remark}
\newenvironment{remark}{\begin{_remark}\rm}{\end{_remark}}

\newtheorem{_example}[theorem]{Example}
\newenvironment{example}{\begin{_example}\rm}{\end{_example}}

\numberwithin{equation}{section}
\numberwithin{table}{section}
\numberwithin{figure}{section}
\renewcommand{\qed}{\hfill {$\Box$}}

\newcommand{\C}{\mathord{\mathbb C}}
\newcommand{\F}{\mathord{\mathbb F}}
\renewcommand{\H}{\mathord{\mathbb H}}
\renewcommand{\P}{\mathord{\mathbb  P}}
\newcommand{\Q}{\mathord{\mathbb  Q}}
\newcommand{\R}{\mathord{\mathbb R}}
\newcommand{\Z}{\mathord{\mathbb Z}}

\newcommand{\AAA}{\mathord{\mathcal A}}
\newcommand{\CCC}{\mathord{\mathcal C}}
\newcommand{\EEE}{\mathord{\mathcal E}}
\newcommand{\FFF}{\mathord{\mathcal F}}
\newcommand{\LLL}{\mathord{\mathcal L}}
\newcommand{\MMM}{\mathord{\mathcal M}}

\newcommand{\PPP}{\mathord{\mathcal P}}
\newcommand{\QQQ}{\mathord{\mathcal Q}}

\newcommand{\TTT}{\mathord{\mathcal T}}
\newcommand{\VVV}{\mathord{\mathcal V}}
\newcommand{\WWW}{\mathord{\mathcal W}}
\newcommand{\XXX}{\mathord{\mathcal X}}

\newcommand{\SSSS}{\mathord{\mathfrak S}}

\newcommand{\maprightsp}[1]{\; \smash{\mathop{\; \longrightarrow \; }\limits\sp{#1}}\; }

\newcommand{\mapdown}{\phantom{\Big\downarrow}\hskip -8pt \downarrow}
\newcommand{\mapdownright}[1]{\mapdown\rlap{$\vcenter{\hbox{$\scriptstyle#1$}}$}}
\newcommand{\mapdownleft}[1]{\llap{$\vcenter{\hbox{$\scriptstyle#1$\;}}$}\mapdown}

\newcommand{\inj}{\hookrightarrow}
\newcommand{\surj}{\mathbin{\to \hskip -7pt \to}}

\newcommand{\set}[2]{\{\, {#1} \mid {#2} \, \}}
\newcommand{\shortset}[2]{\{ {#1} \mid {#2}   \}}
\newcommand{\sethd}[3]{\left\{\;  {#1}\; \left|\;  \vcenter{\hbox{\parbox{#2}{#3}}}\;  \right. \right\}}

\newcommand{\gen}[1]{\langle {#1}  \rangle}

\newcommand{\wt}{\widetilde}

\newcommand{\tensor}{\otimes}

\newcommand{\sprime}{\sp\prime}
\newcommand{\sprimeinv}{\sp{\prime-1}}
\newcommand{\spar}[1]{\sp{(#1)}}
\newcommand{\spprime}{\sp{\prime\prime}}

\newcommand{\sptimes}{\sp{\times}}
\newcommand{\sperp}{\sp{\perp}}

\newcommand{\dual}{\sp{\vee}}

\newcommand{\semidirectproduct}{\rtimes}

\newcommand{\inv}{\sp{-1}}

\newcommand{\Hom}{\mathord{\mathrm{Hom}}}

\newcommand{\GL}{\mathord{\mathrm{GL}}}
\newcommand{\PGU}{\mathord{\mathrm{PGU}}}
\newcommand{\PGL}{\mathord{\mathrm{PGL}}}
\newcommand{\PSL}{\mathord{\mathrm{PSL}}}
\newcommand{\OG}{\mathord{\mathrm{O}}}

\newcommand{\Ker}{\operatorname{\mathrm{Ker}}\nolimits}
\newcommand{\Aut}{\operatorname{\mathrm{Aut}}\nolimits}
\newcommand{\Gal}{\operatorname{\mathrm{Gal}}\nolimits}

\newcommand{\pr}{\mathord{\mathrm{pr}}}
\newcommand{\Sing}{\operatorname{\mathrm{Sing}}\nolimits}
\newcommand{\Spec}{\operatorname{\mathrm{Spec}}\nolimits}
\newcommand{\rank}{\operatorname{\mathrm{rank}}\nolimits}

\newcommand{\mystruthd}[2]{\phantom{\hbox{\vrule  height #1 depth #2}}}

\newcommand{\thePGU}{\PGU_4(\F_9)}
\newcommand{\IV}{{\rm I~\hspace{-1.1mm}V}}

\newcommand{\NS}[1]{S_{#1}}
\newcommand{\PP}[1]{{\mathcal P}_{#1}}
\newcommand{\NN}[1]{N_{#1}}

\newcommand{\intf}[1]{\langle #1\rangle}

\newcommand{\vare}{\varepsilon}

\newcommand{\Stab}{\mathord{\rm Stab}}
\newcommand{\res}{\mathord{\rm res}}

\newcommand{\Triv}{\mathord{\rm Triv}}
\newcommand{\MW}{\mathord{\rm MW}}

\newcommand{\Roots}{{\mathcal R}}

\newcommand{\TL}[1]{T_{#1}}

\newcommand{\QP}{\mathord{\rm QP}}

\newcommand{\wcirc}{\setlength\unitlength{.2truecm}%
\begin{picture}(1.7,1.7)(0,0)%
\put(.9,.5){\circle{.7}}%
\end{picture}}

\newcommand{\PG}{\PPP}
\newcommand{\QPG}{\QQQ}

\newcommand{\discg}[1]{A(#1)}
\newcommand{\discf}[1]{q(#1)}

\newcommand{\latgraph}[1]{\gen{#1}}
\newcommand{\nullintf}{\intf{\phantom{\cdot}, \phantom{\cdot}}}
\newcommand{\OGplus}{\OG^+}

\newcommand{\out}{\mathord{\rm out}}
\newcommand{\inn}{\mathord{\rm inn}}

\newcommand{\YIV}[1]{Y_{\hbox to 2.6mm {\scriptsize\rm IV}, \hskip 1pt #1}}

\newcommand{\dppinvol}{g}

\newcommand{\dcirc}{\setlength\unitlength{.2truecm}%
\begin{picture}(2,1.5)(0,0)%
\put(1,.5){\circle{1}}%
\put(1,.5){\circle{.5}}%
\end{picture}}

\newcommand{\theOGpair}{\OGplus(\NS{3}, \NS{0})}
\newcommand{\dpp}{b}
\newcommand{\gdpp}[1]{g(\dpp_{#1})}

\newcommand{\Cr}{\mathord{\rm {Cr}}}
\newcommand{\prV}{{\mathcal V}}

\newcommand{\PStextplot}{{}}

\newcommand{\Pic}{\mathord{\mathrm {Pic}}}
\newcommand{\PPb}{\mathbf{P}}
\begin{document}
\title[Elliptic modular surface of level $4$]
{The elliptic modular surface of level $4$ \\
and its reduction modulo $3$}
\author{Ichiro Shimada}
\address{Department of Mathematics, 
Graduate School of Science, 
Hiroshima University,
1-3-1 Kagamiyama, 
Higashi-Hiroshima, 
739-8526 JAPAN}
\email{ichiro-shimada@hiroshima-u.ac.jp}
\thanks{Supported by JSPS KAKENHI Grant Number 15H05738, ~16H03926,  and~16K13749}

\begin{abstract}
The elliptic modular surface of level $4$ is 
a complex $K3$ surface with Picard number $20$.
This surface has a model over a number field such that 
its reduction modulo $3$
yields a surface  isomorphic to the Fermat quartic surface in characteristic $3$,
which is supersingular.
The specialization  induces an embedding 
of the N\'eron-Severi lattices.
Using this embedding,
we determine the automorphism group 
of this $K3$ surface over a discrete valuation ring of mixed characteristic 
whose residue field is of characteristic $3$.
\par
The elliptic modular surface of level $4$ has a fixed-point free involution that
gives rise to the Enriques surface
of type IV in  Nikulin--Kondo--Martin's classification of  Enriques surfaces
with finite automorphism group.
We investigate  the specialization of this  involution to characteristic $3$.
\end{abstract}

%\dedicatory{Dedicated to Professor 
%Shigeyuki Kondo on the occasion of his 60th birthday}
\maketitle

\section{Introduction}\label{sec:intro}
Let $R$ be a discrete valuation ring,
and let $\XXX\to \Spec R$ be a smooth proper family of varieties over $R$.
We denote by $X_{\bar{\eta}}$ the geometric generic fiber,
and by $X_{\bar{s}}$ the geometric special fiber.
Let $\Aut(\XXX/R)$ denote the group of automorphisms of $\XXX$ over $\Spec R$.
Then we have natural homomorphisms 
$$
\Aut(X_{\bar{s}})\leftarrow \Aut(\XXX/R)\rightarrow \Aut(X_{\bar{\eta}}).
$$
In this paper,
we calculate the group  $\Aut(\XXX/R)$ 
in the case where  $\XXX$ is a certain natural model 
of the elliptic modular surface 
of level $4$, 
and the special fiber $X_{\bar{s}}$ is its reduction modulo $3$.
%$X_{\bar{\eta}}$ is the elliptic modular surface of level $4$
%and $X_{\bar{s}}$ is its reduction modulo $3$.
In this case, the surfaces $X_{\bar{\eta}}$ and $X_{\bar{s}}$  are $K3$ surfaces, and 
their  automorphism groups 
have been calculated in~\cite{KK2001} and~\cite{KondoShimadaFQ3},
respectively, by Borcherds' method~\cite{Bor1, Bor2}.
This paper gives the first application of 
Borcherds' method to 
the calculation of the automorphism group of a family of $K3$ surfaces.
%The result of this paper shows that Borcherds' method is useful
%not only for the calculation of the automorphism group of a $K3$ surface,
%but also for the calculation of the automorphism group of a family of $K3$ surfaces over a complete discrete valuation ring.
%
\subsection{Elliptic modular surface of level $4$}\label{subsec:ShiodaKK}
The \emph{elliptic modular surface of level $N$} is a natural  compactification of the total space of the universal family
over $\Gamma(N)\backslash \H$ of 
complex elliptic curves
with level $N$ structure,
where $\H\subset \C$ is the upper-half plane and $\Gamma(N)\subset \PSL_2(\Z)$ is
the congruence subgroup of level $N$.
This important class of surfaces was introduced and studied by Shioda~\cite{ShiodaEllMod}.
\par
The elliptic modular surface of level $4$ is a $K3$ surface birational to the surface  defined by the Weierstrass equation
\begin{equation}\label{eq:theellfib4}
Y^2=X(X-1)\left(X-\left(\frac{1}{2}\left(\sigma+\frac{1}{\sigma}\right)\right)^2\right),
\end{equation}
where $\sigma$ is an affine parameter of the base curve 
$\P^1=\overline{\Gamma(4)\backslash \H}$
(see Section 3 in~\cite{ShiodaManifoldsTokyo}).
Shioda~\cite{ShiodaEllMod, ShiodaManifoldsTokyo}
studied  the reduction of this surface in odd characteristics.
On the other hand,
Keum and Kondo~\cite{KK2001} calculated the automorphism group of the elliptic modular surface of level $4$.
\par 
To describe the results of Shioda~\cite{ShiodaEllMod, ShiodaManifoldsTokyo}
and Keum--Kondo~\cite{KK2001}, we fix some notation.
A \emph{lattice} is a free $\Z$-module $L$ of finite rank with a nondegenerate symmetric bilinear form 
$\nullintf\colon L\times L\to \Z$. 
The group of isometries of a lattice $L$ is denoted by $\OG(L)$, which we  let act on $L$ from the \emph{right}.
A lattice $L$ of rank $n$ is said to be \emph{hyperbolic} (resp.~\emph{negative-definite})
if the signature of $L\tensor\R$ is $(1, n-1)$ (resp.~$(0,n)$).
For a hyperbolic lattice $L$,
we denote by  $\OGplus (L)$ the stabilizer subgroup
of a connected component of $\shortset{x\in L\tensor\R}{\intf{x, x}>0}$ in $\OG(L)$. 
Let $Z$ be a smooth projective surface
defined over an algebraically closed field.
We denote by $\NS{Z}$ the lattice of numerical equivalence classes $[D]$ of 
divisors $D$ on $Z$,
and call it the \emph{N\'eron-Severi lattice} of $Z$.
Then  $\NS{Z}$ 
 is hyperbolic by the Hodge index theorem. 
We denote by $ \PP{Z}$ the 
connected component of 
$\shortset{x\in \NS{Z}\tensor\R}{\intf{x, x}>0}$
that contains an ample class.
We then put
$$
\NN{Z}:=\set{x\in  \PP{Z}}{\intf{x, [C]}\ge 0\;\; \textrm{for all curves $C$ on $Z$}}.
$$
We let  the automorphism group $\Aut(Z)$ of $Z$ act on $\NS{Z}$ from the \emph{right} by pull-back of divisors.
Then we have a natural homomorphism
$$
\Aut(Z)\to \Aut(\NN{Z}):=\set{g\in \OGplus(\NS{Z})}{\NN{Z}^g=\NN{Z}}.
$$
For an ample class $h\in \NS{Z}$,
we put
$$
\Aut(Z, h):=\set{g\in \Aut(Z)}{h^g=h},
$$
and call it the \emph{projective automorphism group} of the polarized  surface $(Z, h)$.
\par
Let $k_p$ be an algebraically closed field of characteristic $p\ge 0$.
From now on, we assume that $p\ne 2$.
Let $\sigma\colon X_p\to \P^1$ be the smooth minimal elliptic  surface  defined over $k_p$
 by~\eqref{eq:theellfib4}.
Then $X_p$ is a $K3$ surface.
For simplicity, we use the following notation throughout this paper:
$$
\NS{p}:=\NS{X_p}, \quad \PPP_p:=\PPP_{X_p}, \quad \NN{p}:=\NN{X_p}.
$$
Shioda~\cite{ShiodaEllMod, ShiodaManifoldsTokyo} proved the following:
\begin{theorem}[Shioda~\cite{ShiodaEllMod, ShiodaManifoldsTokyo}]\label{thm:Shioda}
Suppose that $p\ne 2$.
\par
{\rm (1)} The elliptic surface $\sigma\colon X_p\to \P^1$
has exactly $6$ singular fibers.
These singular fibers are located over  $\sigma=0, \pm 1, \pm i, \infty$,
and each of them is  of type ${\rm I}_4$.
The torsion part of the Mordell-Weil group of $\sigma\colon X_p\to \P^1$
is isomorphic to $(\Z/4\Z)^2$.
\par
{\rm (2)}  The Picard number $\rank (\NS{p})$ of $X_p$ is
$$
\begin{cases}
 20 & \textrm{if $p=0$ or $p\equiv 1\bmod 4$}, \\
 22 & \textrm{if $p\equiv 3 \bmod 4$}.
\end{cases}
$$
\par
{\rm (3)} If $k_0=\C$,  the transcendental lattice of the complex $K3$ surface $X_0$ is
\begin{equation*}\label{eq:4004}
\left(\begin{array}{cc} 4 & 0 \\ 0 &4 \end{array}\right).
\end{equation*}
\par
{\rm (4)}  The $K3$ surface $X_3$ is isomorphic to the Fermat quartic surface
$$
F_3 \;\colon \;x_1^4+x_2^4+x_3^4+x_4^4=0
$$
in characteristic $3$.
\end{theorem}
It follows from Theorem~\ref{thm:Shioda}~(3)  and the theorem of Shioda-Inose~\cite{ShiodaInose} 
that, over the complex number field, $X_0$ is isomorphic to the Kummer surface 
associated with $E_{\sqrt{-1}}\times E_{\sqrt{-1}}$,
where $E_{\sqrt{-1}}$ is the  elliptic curve $\C/(\Z\oplus \Z {\sqrt{-1}})$.
(See also Proposition~15 of Barth--Hulek~\cite{BH1985}.)
Therefore the result of Keum--Kondo~\cite{KK2001} contains the calculation of $\Aut(X_0)$.
\begin{definition}
Let $Z$ be a $K3$ surface defined over $k_p$.
A \emph{double-plane polarization}
is a vector $\dpp=[H]\in \NN{Z}\cap\NS{Z}$ with $\intf{\dpp, \dpp}=2$
such that the corresponding complete linear system $|H|$
is base-point free, so that   $|H|$ induces a surjective morphism
$\Phi_\dpp\colon Z\to \P^2$.
Let $\dpp$ be a double-plane polarization, and let $Z\to Z_\dpp\to \P^2$ be 
the Stein factorization of $\Phi_\dpp$.
Then we have a \emph{double-plane involution} $\dppinvol(\dpp)\in \Aut(Z)$
associated with the finite double covering $Z_\dpp\to \P^2$.
Let  $\Sing (\dpp)$ 
denote the singularities of the normal $K3$ surface $Z_\dpp$.
Since  $Z_\dpp$ has only rational double points as its singularities,
we have the \emph{$ADE$-type} of $\Sing (\dpp)$.
\end{definition}
\begin{remark}\label{rem:dppalgo}
Suppose that  an ample class $a\in \NS{Z}$ and a vector $b\in \NS{Z}$ with $\intf{b, b}=2$ are given.
Then we can determine whether  $b$ is a double-plane polarization or not,
and if $b$ is a double-plane polarization,
we can calculate the set  of classes of smooth rational curves
contracted by $\Phi_b\colon Z\to \P^2$,
and compute the matrix representation of the double-plane involution $g(b)\colon Z\to Z$ on $\NS{Z}$.
These algorithms are described in detail in~\cite{ShimadaChar5}~(and also in~\cite{ShimadaSalem}).
They are the key tools of this paper.
\end{remark}
We re-calculated $\Aut(X_0)$ by using these algorithms, and obtained a  generating set of  $\Aut(X_0)$
different from the one given in~\cite{KK2001}.
\begin{theorem}[Keum--Kondo~\cite{KK2001}]\label{thm:KK2001}
There exist an ample class $h_0\in \NS{0}$  of degree $\intf{h_0, h_0}=40$
and four double-plane polarizations 
$\dpp_{80}, \dpp_{112}, \dpp_{296}, \dpp_{688}\in \NS{0}$
such that   $\Aut(X_0)$  is generated by the projective automorphism group 
$ \Aut(X_0, h_0)\cong (\Z/2\Z)^5 \mathord{\,:\,} \SSSS_5$ 
and the  double-plane involutions
$\dppinvol(\dpp_{80}), \dppinvol(\dpp_{112}), 
\dppinvol(\dpp_{296}), \dppinvol(\dpp_{688})$.
\end{theorem}
\begin{table}
$$
\renewcommand{\arraystretch}{1.2}
\begin{array}{ccc}
 \intf{h_0, \dpp_d} & \textrm{$ADE$-type of  $\Sing (\dpp_d)$}  &  d=\intf{h_0, h_0^{g(\dpp_d)}} \\
\hline
 16 & 2 A_3 +3 A_2 +2A_1  & 80 \\
 18 & 4 A_3 +3A_1  & 112 \\
 26 & A_5+2A_4+A_3  & 296 \\
 38 & 2A_7+A_3+A_1 & 688 
\end{array}
$$
\caption{Four double-plane involutions on $X_0$}\label{table:fourdpps}
\end{table}
See Table~\ref{table:fourdpps} for the properties of the double-plane polarizations $\dpp_d$. 
See Proposition~\ref{prop:X0gdpp}  for the geometric meaning of these generators of $ \Aut(X_0)$ 
with respect to the action of $\Aut(X_0)$ on  $\NN{0}$.
In Section~\ref{subsec:AutX0h0},
we also give a detailed description of the finite group $\Aut(X_0, h_0)$
in terms of a certain graph $\LLL_{40}$. 
\begin{remark}
In~\cite{KondoShimadaFQ3},
the automorphism group $\Aut(X_3)\cong \Aut(F_3)$ of
 the Fermat quartic surface $F_3$ in characteristic $3$ was calculated (see Theorem~\ref{thm:KondoShimadaFQ3}).
This calculation also plays an important role in the proof of our main results.
\end{remark}
\subsection{Main results}
In~\cite{KK2001}~and~\cite{KondoShimadaFQ3}, the following was proved, and 
hence, from now on,  we regard $\Aut(X_0)$ as a subgroup of $\OGplus (\NS{0})$ and 
$\Aut(X_3)$ as a subgroup of $\OGplus (\NS{3})$.
\begin{proposition}\label{prop:injpsi}
In each case of $X_0$ and $X_3$,
the action of the automorphism group 
on the N\'eron-Severi lattice is faithful.
\qed
\end{proposition}
\par
Let $R$ be a discrete valuation ring whose fraction field $K$ is of characteristic $0$
and whose residue field $k$ is of characteristic $3$.
Suppose that $\sqrt{-1}\in R$.
In Section~\ref{subsec:geomQPcovering}, 
we construct explicitly 
a smooth family of $K3$ surfaces 
$\XXX\to \Spec R$
over $R$ 
such that the geometric generic fiber $\XXX\tensor_{R}\bar{K}$
is isomorphic to $X_0$ and 
the geometric special fiber $\XXX\tensor_{R}\bar{k}$
is isomorphic to $X_3$.
The construction of this model $\XXX$ is natural in the sense 
that it uses the inherent elliptic fibration of $X_0$. 
Note that the model of $X_0$ over $R$ is not unique,
and that the main results on $\Aut(\XXX/R)$ below 
may depend on the choice of the model.
\par
%(see Section~\ref{subsec:geomQPcovering}).
By Proposition 3.3 of Maulik and Poonen~\cite{MP2012},
the specialization from $\XXX\tensor_{R} K$ to $\XXX\tensor_{R}k$  
gives rise to a homomorphism   
\begin{equation*}
\rho\colon \NS{0}\to \NS{3}.
\label{eq:rho}
\end{equation*}
In Section~\ref{subsec:LLL40},
we give an explicit description of $\rho$.
It turns out that $\rho$ is a primitive embedding of lattices.
We  regard $\NS{0}$ as a sublattice of $\NS{3}$ by $\rho$,
and  put
$$
\theOGpair:=\set{g\in \OGplus(\NS{3})}{\NS{0}^g=\NS{0}}.
$$
Then we have a natural restriction homomorphism
$$
\tilde{\rho}\colon \theOGpair\to \OGplus(\NS{0}).
$$
The main results of this paper are as follows:
\begin{theorem}\label{thm:main}
The restriction of $\tilde{\rho}$ to $\theOGpair\cap \Aut(X_3)$
induces an injective homomorphism
$$
\tilde{\rho}|_{\Aut}\colon  \theOGpair\cap \Aut(X_3) \inj \Aut(X_0).
$$
The image of $\tilde{\rho}|_{\Aut}$ is generated by the finite subgroup  $ \Aut(X_0, h_0)$
and the two  double-plane involutions $\dppinvol(\dpp_{112}), \dppinvol(\dpp_{688})$.
The other double-plane involutions $\dppinvol(\dpp_{80})$ and $\dppinvol(\dpp_{296})$ do not belong to the image of $\tilde{\rho}|_{\Aut}$.
\end{theorem}
Let $R\sprime$ be a finite extension of $R$,
and let $\XXX\sprime:=\XXX\tensor_R R\sprime \to \Spec R\sprime$
be the pull-back of $\XXX\to \Spec R$.
We have a natural embedding
$\Aut(\XXX/R)\inj \Aut(\XXX\sprime/R\sprime)$.
We put
$$
\Aut(\overline{\XXX/R}):=\mathord{\rm colim}_{R\sprime} \Aut(\XXX\sprime/R\sprime).
$$
Let $\res_3\colon \Aut(\overline{\XXX/R})\to \Aut(X_3)$ and $\res_0\colon \Aut(\overline{\XXX/R})\to \Aut(X_0)$
denote the restriction homomorphisms.
It is obvious that $\res_0$ is injective, and  that the following diagram commutes.
\begin{equation}\label{eq:triangle}
\begin{array}{cccccc}
&\Aut(\overline{\XXX/R})&\\
\hskip 3cm\hbox{\scriptsize  $\res_3$} \swarrow && \searrow \hbox{\scriptsize  $\res_0$}\hskip  4mm\\
 \theOGpair\cap \Aut(X_3)  & \vbox{\hbox{\scriptsize$\tilde{\rho}|_{\Aut}$}\vskip 5pt \hbox{$\inj$} }& \hskip  2mm \Aut(X_0) &\phantom{aaaaa}
\end{array}
\end{equation}
\begin{theorem}\label{thm:mainR}
The image of $\res_0$
is equal to the image of $\tilde{\rho}|_{\Aut}$.
\end{theorem}
Thus we have obtained a set of generators of $\Aut(\overline{\XXX/R})$.
\subsection{Enriques surfaces}
By Nikulin~\cite{Nikulin84} and Kondo~\cite{KondoFinite},
the complex Enriques surfaces with finite automorphism group
are classified,
and this classification is extended to Enriques surfaces in odd characteristics
by Martin~\cite{Martin}.
The Enriques surfaces in characteristic $\ne 2$ with finite automorphism group
are divided into seven classes ${\rm I}$-${\rm VII}$.
In this paper, we concentrate on the Enriques surface of type~$\IV$.
\begin{definition}
A fixed-point free involution of a $K3$ surface 
 in characteristic $\ne 2$
 is called an \emph{Enriques involution}.
An Enriques surface $Y$ in characteristic $\ne 2$
is \emph{of type~$\IV$} if $\Aut(Y)$ is of order $320$.
An Enriques involution  of a $K3$ surface
is \emph{of type~$\IV$}
if the quotient Enriques surface is of type~$\IV$.
\end{definition}
\begin{proposition}[Kondo~\cite{KondoFinite}, Martin~\cite{Martin}]
In each characteristic $\ne 2$,
an Enriques surface of type~$\IV$ exists and is unique up to isomorphism.
There exist exactly $20$ smooth rational curves on an Enriques surface of type~$\IV$.
\qed
\end{proposition}
Let $\YIV{p}$ denote an Enriques surface of type~$\IV$ in characteristic $p\ne 2$.
Kondo~\cite{KondoFinite} showed that
the covering $K3$ surface of $\YIV{0}$ is isomorphic to $X_0$. 
\begin{proposition}\label{prop:Enriques}
There exist exactly $6$ Enriques involutions
in the projective automorphism group $\Aut(X_0, h_0)$.
These $6$ Enriques involutions are conjugate in $\Aut(X_0, h_0)$,
and hence 
the corresponding Enriques surfaces are isomorphic to each other.
All of them are of type~$\IV$.
\end{proposition}
By Theorem~\ref{thm:main}, these $6$ Enriques involutions in  
$\Aut(X_0, h_0)$  specialize to  involutions of $X_3$.
\begin{theorem}\label{thm:main2}
Let  $\vare_3\in \Aut(X_3)$ be an involution  that is mapped to 
an Enriques involution  in  $\Aut(X_0, h_0)$  by $\tilde{\rho}|_{\Aut}$.
Then $\vare_3$ is an Enriques involution of type~$\IV$,
and the pull-backs of the $20$ smooth rational curves on 
$X_3/\gen{\vare_3}\cong \YIV{3}$ 
by the quotient morphism $X_3\to X_3/\gen{\vare_3}$   are lines of 
the Fermat quartic surface $F_3\cong X_3$.
\end{theorem}
\begin{figure}
\begin{center}
%\hskip -2cm 
\includegraphics{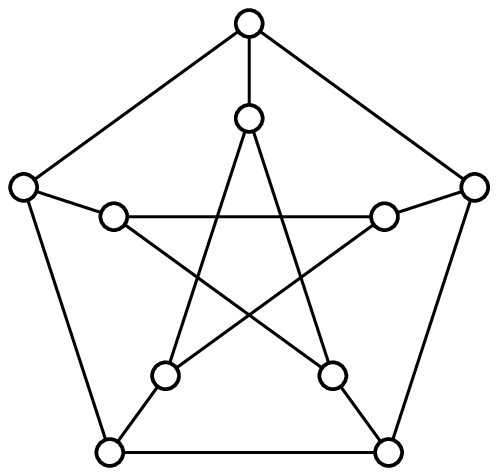}%
\end{center}
\caption{Petersen graph }\label{fig:Petersen}
\end{figure}
During the investigation,
we have come to notice that the geometry of  $X_p$ and $\YIV{p}$
is closely related to the \emph{Petersen graph} (Figure~\ref{fig:Petersen}). 
See Section~\ref{sec:S0S3}  for this relation.
As a by-product,  we see that the 
dual graph of the $20$ smooth rational curves on $\YIV{p}$
is as in Figure~\ref{fig:Enrstar}.
Compare Figure~\ref{fig:Enrstar} with  the picturesque but complicated  figure 
of Kondo~(Figure 4.4 of~\cite{KondoFinite}).
%These two figures  in fact depict the same configuration.
%
\par
It has been observed that 
the Petersen graph is related to various $K3$/Enriques surfaces.
See, for example, Vinberg~\cite{Vinberg1983} for the relation
with the singular $K3$ surface with the transcendental lattice of discriminant $4$.
See also  Dolgachev-Keum~\cite{DK2002} and 
Dolgachev~\cite{Dolgachev2018} for the relation
with  Hessian quartic surfaces and associated Enriques surfaces.
\begin{figure}
\includegraphics{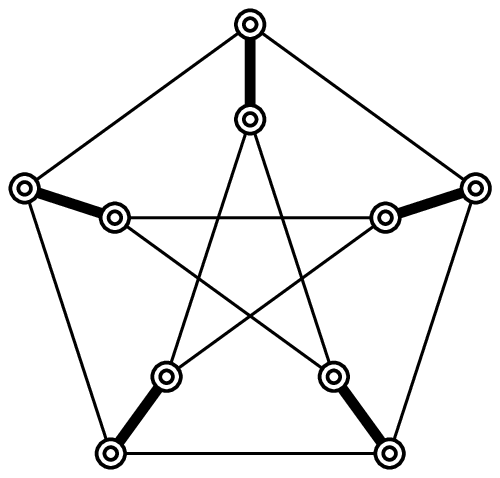}
\raise 2.4cm \hbox{\parbox{6cm}{\small Here $\dcirc$ is a pair of disjoint 
smooth rational curves,
and $
\setlength\unitlength{.2truecm}
\begin{picture}(5,1)(0,0)
\put(1,.5){\circle{1}}
\put(1,.5){\circle{.5}}
\put(1.5,.5){\line(1,0){2}}
\put(4,.5){\circle{1}}
\put(4,.5){\circle{.5}}
\end{picture}
$ 
means that the two pairs of 
smooth rational curves intersect as 
\begin{center}
\includegraphics{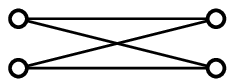} \hskip -3mm \raise 1mm \hbox{$,$} 
\end{center}
whereas 
$
\setlength\unitlength{.2truecm}
 \linethickness{.6mm}
\begin{picture}(5,2)(0,0)
\put(1,.5){\circle{1}}
\put(1,.5){\circle{.5}}
\put(1.5,.5){\line(1,0){2}}
\put(4,.5){\circle{1}}
\put(4,.5){\circle{.5}}
\end{picture}
$ 
means  that the two pairs  intersect as 
\begin{center}
\includegraphics{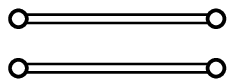} \hskip -3mm \raise 1mm \hbox{$.$}%
\end{center}
}}
\caption{Smooth rational curves on $\YIV{p}$}\label{fig:Enrstar}
\end{figure}
\subsection{Plan of the paper}
In Section~\ref{sec:S0S3},
we present a precise description of the embedding  $\rho\colon \NS{0}\inj \NS{3}$.
First we introduce the notion of $\QP$-graphs.
Then, using an isomorphism $X_3\cong F_3$ given by Shioda~\cite{ShiodaManifoldsTokyo},
we show that $\NS{0}$ is a lattice obtained from a $\QP$-graph,
and calculate the embedding  $\rho\colon \NS{0}\inj \NS{3}$ explicitly.
An elliptic modular surface of level $4$ over a discrete valuation ring 
is constructed, and the relation with the Petersen graph is explained geometrically.
In Section~\ref{sec:Borcherds}, 
we review the method of Borcherds~\cite{Bor1, Bor2}
to calculate the orthogonal group of an even hyperbolic lattice,
and fix terminologies about  \emph{chambers}.
The application of this method to $K3$ surfaces is also explained.
In Section~\ref{sec:BorcherdsX0X3}, 
we review the results of~\cite{KK2001} for $\Aut(X_0)$ and 
of~\cite{KondoShimadaFQ3} for  $\Aut(X_3)$.
Using the chamber tessellations of $\NN{0}$ and  $\NN{3}$ obtained in these works, 
we give a proof of Theorems~\ref{thm:main}~and~\ref{thm:mainR} 
in Section~\ref{sec:proof}.
In Section~\ref{sec:Enriques}, 
we investigate  Enriques involutions of $X_0$ and  $X_3$.
\par
In this paper, 
we fix bases of lattices
and reduce proofs of our results to simple computations of vectors and matrices.
Unfortunately, these  vectors and matrices are too large to be presented in the paper.
We refer the reader to the author's web site~\cite{thecompdata}
for this data.
In the computation, we used {\tt GAP}~\cite{GAP}.
\par
Thanks are due to
Professors I.~Dolgachev, G.~van der Geer, S.~Kondo, 
Y.~Matsumoto, S.~Mukai, H.~Ohashi, T.~Shioda, and T.~Terasoma.
In particular,  the contents of Section~\ref{subsec:geomQPcovering} are  
obtained through discussions with S.~Mukai and T.~Terasoma.
Thanks are also due to the referees of the first and second version of this paper
for their many comments and suggestions.
In particular,  the contents of Section~\ref{subsec:anothermodel} are  
suggested by one of the referees.
\section{The lattices  $\NS{0}$ and $\NS{3}$}\label{sec:S0S3}
\subsection{Graphs and lattices}\label{subsec:graphsandlattices}
First we fix terminologies and notation about graphs and lattices.
\par
A \emph{graph} (or more precisely, a \emph{weighted graph}) is a pair $(V, \eta)$,
where $V$ is a set of \emph{vertices} and $\eta$ is a map from the set $V \choose 2$ of non-ordered pairs of distinct elements of $V$ to $\Z_{\ge 0}$.
When the image of $\eta$ is contained in $\{0, 1\}$,
we say that  $(V, \eta)$ is \emph{simple},
and denote it by $(V, E)$, where $E=\eta\inv (1)$ is the set of \emph{edges}.
Let $\Gamma=(V, E)$ and $\Gamma\sprime=(V\sprime, E\sprime)$ be simple graphs.
A \emph{map $\gamma\colon \Gamma\to \Gamma\sprime$ of  simple graphs} 
is a pair of maps $\gamma_{V}\colon V\to V\sprime$ and $\gamma_E\colon E\to E\sprime$
such that, for all $\{v, v\sprime\}\in E$,
we have $\gamma_E(\{v, v\sprime\})=\{\gamma_V(v), \gamma_V(v\sprime)\}\in E\sprime$.
A graph is depicted by indicating each vertex by $\wcirc$ and $\eta(\{v, v\sprime\})$ by the number of line segments
connecting $v$ and $v\sprime$.
The \emph{Petersen graph} $\PG=(V_{\PG}, E_{\PG})$ is the simple graph given by Figure~\ref{fig:Petersen}.
It is well-known that the automorphism group 
$\Aut(\PG)$ of $\PG$ is isomorphic to the symmetric group $\SSSS_5$.
\par
A submodule $M$ of a free $\Z$-module $L$ is \emph{primitive} if $L/M$ is torsion free.
A nonzero vector $v$ of $L$ is \emph{primitive} if $\Z v\subset L$ is primitive.
\par 
Let $L$ be a lattice.
We say that  $L$ is \emph{even} if $\intf{x, x}\in 2\Z$ for all $x\in L$.
The \emph{dual lattice} of $L$ is the free $\Z$-module $L\dual:=\Hom(L,\Z)$,
into which $L$ is embedded by $\nullintf$. 
Hence we have $L\dual\subset L\tensor{\Q}$.
The \emph{discriminant group} $\discg{L}$ is 
 the finite abelian group $L\dual/L$.
We say that $L$ is \emph{unimodular} if $A(L)$ is trivial.
\par
With a graph $\Gamma=(V, \eta)$ with $|V|<\infty$,
we associate an even lattice $\latgraph{\Gamma}$  as follows. 
Let $\Z^V$ be the $\Z$-module freely generated by the elements of  $V$.
We define a symmetric bilinear form $\nullintf$ on $\Z^V$ by
$$
\intf{v, v\sprime}=\begin{cases}
-2 & \textrm{if $v= v\sprime$}, \\
\eta(\{v, v\sprime\}) & \textrm{if $v\ne v\sprime$}.
\end{cases}
$$
Let $\Ker \nullintf \subset \Z^V$ denote the submodule 
$\shortset{x\in \Z^V}{\intf{x, y}=0\;\textrm{for all}\; y\in \Z^V}$.
Then $\latgraph{\Gamma}:=\Z^V/\Ker \nullintf$ has 
a natural structure of an even lattice.
\par
Suppose that $Z$ is a $K3$ surface or an Enriques surface
defined over an algebraically closed field.
Let $\LLL$ be a set of smooth rational curves on $Z$.
Then the mapping $C\mapsto [C]$ embeds 
$\LLL$ into the N\'eron-Severi lattice $\NS{Z}$ of $Z$.
The \emph{dual graph} of $\LLL$ is the graph $(\LLL, \eta)$,
where $\eta(\{C_1, C_2\})$ is 
the intersection number of two distinct  curves $C_1, C_2\in \LLL$.
By abuse of notation, we sometimes use $\LLL$ to denote 
the dual graph $(\LLL, \eta)$ or the image of the embedding $\LLL\inj \NS{Z}$.
Then the  even lattice $\latgraph{\LLL}$ constructed from the dual graph of $\LLL$
is canonically identified with the sublattice of $\NS{Z}$
generated by $\LLL\subset \NS{Z}$,
because every smooth rational curve on $Z$ has self-intersection number $-2$.
\begin{example}
Let $\Gamma$ be the graph given by Figure~\ref{fig:Enrstar}.
Then $\gen{\Gamma}$ is an even  hyperbolic lattice of rank $10$ with $\discg{\gen{\Gamma}}\cong (\Z/2\Z)^2$.
Since the N\'eron-Severi lattice of  an Enriques surface  is unimodular of rank $10$, 
the classes of $20$ smooth rational curves on $\YIV{p}$
generate a sublattice of index $2$ in the N\'eron-Severi lattice.
\end{example}
\subsection{$\QP$-graph}
We introduce the notion of \emph{$\QP$-graphs},
where $\QP$ stands for a \emph{quadruple covering of the Petersen graph}.
In the following, 
a \emph{quadrangle} means the simple graph
$\mystruthd{15pt}{12pt}\raise -.25cm \hbox{
\setlength\unitlength{.6truecm}
\begin{picture}(1,1)(.2,.0)
\put(0,0){\circle{.3}} 
\put(0,1){\circle{.3}} 
\put(1,0){\circle{.3}} 
\put(1,1){\circle{.3}} 
\put(1,1){\circle{.3}} 
\put(.15,0){\line(1,0){0.7}}
\put(.15,1){\line(1,0){0.7}}
\put(0,.15){\line(0,1){0.7}}
\put(1,.15){\line(0,1){0.7}}
\end{picture}
}$.
\begin{definition}\label{def:QPG}
A \emph{$\QP$-graph} is a pair $(\QPG, \gamma)$
of a simple graph $\QPG=(V_{\QPG}, E_{\QPG})$
and a map $\gamma\colon \QPG\to \PG$
to the Petersen graph with the following properties.
\begin{enumerate}[{\rm (i)}]
\item The map $\gamma_{V}\colon V_{\QPG}\to V_{\PG}$ is surjective,
and every fiber of $\gamma_{V}$ is of size $4$.
\item
For any edge $e$
of $\PG$, the subgraph 
$(\gamma_V\inv (e),  \gamma_E\inv (\{e\}))$
of $\QPG$ is isomorphic to the disjoint union of two quadrangles.
$$
\includegraphics{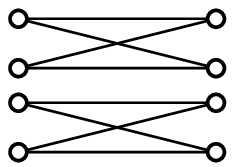}%
$$
\item
Any two distinct quadrangles in $\QPG$ have at most one common  vertex.
\end{enumerate}
A map $\gamma\colon \QPG\to \PG$ satisfying conditions (i)-(iii) is 
called a \emph{$\QP$-covering map}.
Two $\QP$-graphs $(\QPG, \gamma)$ and $(\QPG\sprime, \gamma\sprime)$
are said to be \emph{isomorphic} if there exists an isomorphism $h\colon \QPG\to \QPG\sprime$
such that $\gamma\sprime\circ h=\gamma$.
\end{definition}
%
%We first prove the following purely graph-theoretic proposition:
%
\begin{proposition}\label{prop:2QPGs}
Up to isomorphism,
there  exist exactly two $\QP$-graphs $(\QPG_0, \gamma_0)$ and $(\QPG_1, \gamma_1)$.
The even lattices $\latgraph{\QPG_0}$ and $\latgraph{\QPG_1}$ are hyperbolic
of rank $20$.
The discriminant group $\discg{\latgraph{\QPG_0}}$ of $\latgraph{\QPG_0}$  is isomorphic to $(\Z/2\Z)^2$,
whereas 
$\discg{\latgraph{\QPG_1}}$  
is isomorphic to $(\Z/4\Z)^2$.
\end{proposition}
\begin{proof}
We enumerate  all isomorphism classes of $\QP$-graphs. 
Let $\Delta$ be the set of 
ordered pairs $[\{i_1, i_2\}, \{i_3, i_4\}]$
of non-ordered pairs 
of elements of $\{1,2,3,4\}$ such that $\{i_1, i_2, i_3, i_4\}=\{1,2,3,4\}$.
We have $|\Delta|=6$.
Let $\TTT(\Delta)$ be the set of ordered triples
$[\delta_1, \delta_2, \delta_3]$ of elements of $\Delta$ such that, 
if $\mu\ne\nu$,
then $\delta_{\mu}=[\{i_1, i_2\}, \{i_3, i_4\}]$
and $\delta_{\nu}=[\{i\sprime_1, i\sprime_2\}, \{i\sprime_3, i\sprime_4\}]$
satisfy $|\{i_1, i_2\}\cap \{i_1\sprime, i_2\sprime\}|=1$.
%(and hence we also have  $|\{i_1, i_2\}\cap \{i_3\sprime, i_4\sprime\}|=1$).
Then we have $|\TTT(\Delta)|=48$.
The following facts can be easily verified.
\begin{enumerate}[(a)]
\item The natural action on $\TTT(\Delta)$ of the full permutation group $\SSSS_4$ of  $\{1,2,3,4\}$ 
decomposes $\TTT(\Delta)$ into two orbits $o_1$ and $o_2$ of size $24$.
\item For any triple $[\delta_1, \delta_2, \delta_3]\in \TTT(\Delta)$ 
and any permutation $\mu, \nu, \rho$ of $1, 2, 3$,
the triple $[\delta_{\mu}, \delta_{\nu}, \delta_{\rho}]$
belongs to the same orbit as $[\delta_1, \delta_2, \delta_3]$.
\item For $\delta=[\{i_1, i_2\}, \{i_3, i_4\}]\in \Delta$,
we put $\bar{\delta}:=[\{i_3, i_4\}, \{i_1, i_2\}]\in \Delta$.
Then $[\delta_1, \delta_2, \delta_3]\in \TTT(\Delta)$ 
and  $[\delta_1, \delta_2, \bar{\delta}_3]\in \TTT(\Delta)$ belong to different orbits.
\end{enumerate}
Let $\psi$ be a map from the set $V_{\PG}$ of vertices of $\PG$  to the set $\{o_1, o_2\}$ of the orbits.
We construct a $\QP$-graph $(\QPG_{\psi}, \gamma_{\psi})$ 
with the set of vertices 
$$
V_{\QPG}:=V_{\PG}\times \{1,2,3,4\}
$$
as follows.
For each vertex $v\in V_{\PG}$,
we choose an element $[\delta_1, \delta_2, \delta_3]$
from the orbit $\psi(v)$,
 choose an ordering $e_1, e_2, e_3$ 
on the three edges of $\PG$ emitting from $v$,
and assign $\delta_i$ to the pair $(v, e_i)$ for $i=1,2,3$.
Let $e=\{v, v\sprime\}$ be an edge of $\PG$.
Suppose that $\delta=[\{i_1, i_2\}, \{i_3, i_4\}]$ is assigned to $(v, e)$
and $\delta\sprime=[\{i\sprime_1, i\sprime_2\}, \{i\sprime_3, i\sprime_4\}]$
is assigned to $(v\sprime, e)$.
Then the edges of $\QPG_{\psi}$ lying over the edge $e$ of $\PG$  are the following $8$ edges.
$$
\includegraphics{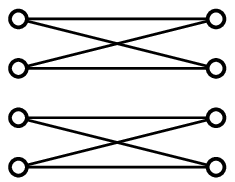}%
\renewcommand{\PStextplot}[3]%
{\rlap{\hskip -85.3727945361pt \hbox{\hskip #1pt%
\raise #2pt \hbox{#3}}}}%%
\PStextplot{-17.07455890722}{2.84575981787}{{\scriptsize $(v, i_4)$}}%%
\PStextplot{-17.07455890722}{17.07455890722}{{\scriptsize $(v, i_3)$}}%%
\PStextplot{-17.07455890722}{31.30335799657}{{\scriptsize $(v, i_2)$}}%%
\PStextplot{-17.07455890722}{45.53215708592}{{\scriptsize $(v, i_1)$}}%%
\PStextplot{79.68127490035999}{2.84575981787}{{\scriptsize $(v\sprime, i\sprime_4)$}}%%
\PStextplot{79.68127490035999}{17.07455890722}{{\scriptsize $(v\sprime, i\sprime_3)$}}%%
\PStextplot{79.68127490035999}{31.30335799657}{{\scriptsize $(v\sprime, i\sprime_2)$}}%%
\PStextplot{79.68127490035999}{45.53215708592}{{\scriptsize $(v\sprime, i\sprime_1)$}}%%

$$
Let $\gamma_{\psi}\colon \QPG_{\psi}\to \PG$ be obtained 
from  the first projection $V_{\QPG}\to V_{\PG}$.
Then $(\QPG_{\psi},  \gamma_{\psi})$ is a $\QP$-graph.
The isomorphism class of $(\QPG_{\psi}, \gamma_{\psi})$ is independent of the choice
of a representative $[\delta_1, \delta_2,\delta_3]$
of each  orbit $\psi(v)$ and the choice of
 the ordering 
of the edges emitting from each vertex of $\PG$.
Indeed, changing these choices merely amounts to relabeling 
the vertices in each fiber of  the first projection $V_{\QPG}\to V_{\PG}$
(see fact (b)).
It is also obvious that every 
 $\QP$-graph
 is isomorphic to 
$(\QPG_{\psi}, \gamma_{\psi})$ 
  for some $\psi\colon V_{\PG}\to \{o_1, o_2\}$.
 \par
 For an orbit $o\in \{o_1, o_2\}$,
 let $\bar{o}$ denote the other orbit; $\{o_1, o_2\}=\{o, \bar{o}\}$.
Let  $\psi\colon V_{\PG}\to \{o_1, o_2\}$ be a map, and 
let $e=\{v, v\sprime\}$ be an edge of $\PG$.
We define $\psi\sprime\colon V_{\PG}\to \{o_1, o_2\}$ by 
$\psi\sprime(v):=\overline{\psi(v)}, \psi\sprime(v\sprime):=\overline{\psi(v\sprime)}$ and 
$\psi\sprime(v\spprime):=\psi(v\spprime)$ for all $v\spprime \in V_{\PG}\setminus\{v, v\sprime\}$.
Then $(\QPG_{\psi}, \gamma_{\psi})$ and
$(\QPG_{\psi\sprime}, \gamma_{\psi\sprime})$ are isomorphic.
(See the picture below and fact (c).)
$$
\includegraphics{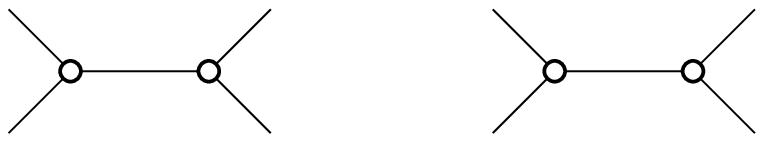}%
\renewcommand{\PStextplot}[3]%
{\rlap{\hskip -318.72509960144pt \hbox{\hskip #1pt%
\raise #2pt \hbox{#3}}}}%%
\PStextplot{77.68924302785099}{21.912350597599}{{\scriptsize $e$}}%%
\PStextplot{217.131474103481}{21.912350597599}{{\scriptsize $e$}}%%
\PStextplot{47.80876494021599}{28.4860557768787}{{\scriptsize $v$}}%%
\PStextplot{187.250996015846}{28.4860557768787}{{\scriptsize $v$}}%%
\PStextplot{104.5816733067225}{28.4860557768787}{{\scriptsize $v\sprime$}}%%
\PStextplot{244.0239043823525}{28.4860557768787}{{\scriptsize $v\sprime$}}%%
\PStextplot{64.7410358565425}{33.864541832653}{{\scriptsize $\delta_3$}}%%
\PStextplot{204.1832669321725}{33.864541832653}{{\scriptsize $\bar{\delta}_3$}}%%
\PStextplot{86.65338645414148}{33.4661354581512}{{\scriptsize $\delta\sprime_3$}}%%
\PStextplot{226.0956175297715}{33.4661354581512}{{\scriptsize $\bar{\delta}\sprime_3$}}%%
\PStextplot{35.856573705162}{39.84063745018}{{\scriptsize $\delta_1$}}%%
\PStextplot{175.298804780792}{39.84063745018}{{\scriptsize $\delta_1$}}%%
\PStextplot{115.537848605522}{39.84063745018}{{\scriptsize $\delta\sprime_1$}}%%
\PStextplot{254.980079681152}{39.84063745018}{{\scriptsize $\delta\sprime_1$}}%%
\PStextplot{35.856573705162}{19.92031872509}{{\scriptsize $\delta_2$}}%%
\PStextplot{175.298804780792}{19.92031872509}{{\scriptsize $\delta_2$}}%%
\PStextplot{115.537848605522}{19.92031872509}{{\scriptsize $\delta\sprime_2$}}%%
\PStextplot{254.980079681152}{19.92031872509}{{\scriptsize $\delta\sprime_2$}}%%

$$
Hence
the isomorphism class of $(\QPG_{\psi}, \gamma_{\psi})$ 
depends only on $|\psi\inv (o_1)|\bmod 2$.
We denote by $(\QPG_0, \gamma_0)$ the $\QP$-graph $(\QPG_{\psi}, \gamma_{\psi})$
with $|\psi\inv (o_1)|\equiv 0\bmod 2$,
and by $(\QPG_1, \gamma_1)$ the $\QP$-graph $(\QPG_{\psi}, \gamma_{\psi})$
with $|\psi\inv (o_1)|\equiv 1\bmod 2$.
Since we have constructed $\QPG_0$ and $\QPG_1$ explicitly,
the assertions on $\latgraph{\QPG_0}$ and $\latgraph{\QPG_1}$ can be proved by direct computation.
\end{proof}
%
%
%The geometric meaning of these propositions will be clear in Sections~\ref{subsec:geomQPcovering}~and~\ref{subsec:AutX0h0}.
%
%
\begin{proposition}\label{prop:AutQPG}
Let $(\QPG, \gamma)$ be a $\QP$-graph.
Each automorphism $g\in \Aut(\QPG)$ maps every fiber of $\gamma_V\colon V_{\QPG}\to V_{\PG}$
to a fiber of $\gamma_V$,
and hence induces  $\bar{g}\in\Aut(\PG)$ such that
$\bar{g}\circ \gamma=\gamma\circ g$.
The mapping $g\mapsto \bar{g}$ 
gives a surjective homomorphism
$$
\Aut(\QPG)\to \Aut(\PG)\cong \SSSS_5, 
$$
and its kernel is isomorphic to $(\Z/2\Z)^6 $.
\end{proposition}
\begin{proof}
Since $\PG$ does not contain a quadrangle, 
every quadrangle of $\QPG$ is mapped to  an edge of $\PG$ by $\gamma$.
Hence two distinct  vertices $v, v\sprime$ of $\QPG$ are mapped to the same vertex of $\PG$ by $\gamma$
if and only if $\{v, v\sprime\}$ is not an edge of $\QPG$ 
and there exists a quadrangle of $\QPG$ containing $v$ and $v\sprime$.
Thus the first assertion follows.
We  make the  complete list of elements of $\Aut(\QPG)$ by computer, and verify 
the  assertion on 
$\Aut(\QPG)\to \Aut(\PG)$.
\end{proof}
\begin{corollary}\label{cor:uniqueQPgamma}
A $\QP$-covering map $\gamma\colon \QPG\to \PG$ from the graph $\QPG$  is 
unique up to the action of $\Aut(\PG)$.
\qed
\end{corollary}
\subsection{The configurations $\LLL_{40}$ and $\LLL_{112}$}\label{subsec:LLL40}
In this section,
following the argument of  Shioda~\cite{ShiodaManifoldsTokyo}, 
we describe the N\'eron-Severi lattices $\NS{0}$ of $X_0$ and $\NS{3}$ of $X_3$,
and investigate   the embedding $\rho\colon \NS{0}\inj \NS{3}$
induced by the specialization of $X_0$ to $X_3$.
%We mainly follow the argument of Shioda~\cite{ShiodaManifoldsTokyo}.
%A more explicit description of the specialization
%will be given in Section~\ref{subsec:geomQPcovering}.
\par
By Theorem~\ref{thm:Shioda}~(1),
we have a distinguished set of 
$$
6\times 4+4^2=40
$$
smooth rational curves on $X_p$,
where the $6\times 4$ curves 
are the  irreducible components of the $6$ singular fibers of $\sigma\colon X_p\to \P^1$
and the $4^2$ curves are the torsion sections of the Mordell-Weil group.
We denote the configuration of these smooth rational curves by $\LLL_{40, p}$, or 
simply by $\LLL_{40}$.
The specialization of $X_0$ to $X_p$ gives a  bijection 
from $\LLL_{40, 0}$  to $\LLL_{40, p}$,
because the specialization preserves the elliptic  fibration $\sigma\colon X_p\to \P^1$
and its zero section.
This bijection is obviously 
compatible with the specialization homomorphism $S_0\to S_p$.
\par 
The set of lines on the Fermat quartic surface $F_3$ in characteristic $3$
has been studied classically by Segre~\cite{Segre}.
The surface $F_3\subset \P^3$ contains exactly $112$ lines,
and every line on $F_3$ is defined over 
the finite field $\F_9$.
We denote by $\LLL_{112}$ the set of these lines.
We  can easily  make the list of defining equations of all lines on $F_3$,
and calculate the dual graph of $\LLL_{112}$.
It is also known~(\cite{KondoShimadaFQ3}) that the classes of 
$22$ lines appropriately chosen from $\LLL_{112}$  form a basis of  $\NS{F_3}\cong \NS{3}$.
Fixing a basis of  $\NS{3}$,
we can express all classes of lines as integer vectors of length $22$
(see~\cite{thecompdata}).
\par
We show that the specialization of $X_0$ to $X_3\cong F_3$
induces an embedding 
$$
\rho_{\LLL}\colon \LLL_{40}\inj \LLL_{112}
$$
of configurations. 
We recall  the construction of the isomorphism $X_3\cong F_3$ by
Shioda~\cite{ShiodaManifoldsTokyo}.
Let 
$\sigma_F\colon F_3 \to \P^1$
be the morphism 
defined by
\begin{equation}\label{eq:sigmaF}
\sigma_F\;\colon\; [x_1: x_2: x_3: x_4]\;\mapsto\; [x_3^2-  i\,x_4^2: x_1^2+i \,x_2^2]=[-x_1^2+i\,x_2^2: x_3^2+i\, x_4^2], 
\end{equation}
where $i=\sqrt{-1}\in \F_9$. 
The generic fiber of $\sigma_F$ is a curve of genus $1$, and 
$\sigma_F$ has a section (see the next paragraph).
Hence the generic fiber of $\sigma_F$ is isomorphic to its Jacobian,
which is defined by the equation~\eqref{eq:theellfib4} by 
the result of Ba\v smakov and Faddeev~\cite{BF1959}. 
Therefore  $\sigma_F\colon F_3\to \P^1$
is isomorphic to $\sigma\colon X_3\to \P^1$ over $\P^1$.
\begin{remark}
In characteristic $0$,
the morphism~\eqref{eq:sigmaF} with $i\in \C$ from the Fermat quartic surface to $\P^1$  has no sections.
\end{remark}
Using the defining equations of lines and the vector representations of their classes,
we confirm the following facts.
These facts make the isomorphism between $\sigma_F\colon F_3\to \P^1$
and $\sigma\colon X_3\to \P^1$  over $\P^1$ more explicit.
There exist exactly $6\times 4$ lines on $F_3$ that are contracted to points by $\sigma_F$.
These $24$ lines  form, of course, a configuration of $6$ disjoint quadrangles.
Moreover, there  exist exactly $64$ lines on $F_3$ that are mapped to $\P^1$ isomorphically by $\sigma_F$.
Let $z_F\in \LLL_{112}$ be one of these $64$ sections of $\sigma_F$.
To be explicit, we choose the following line as $z_F$. (See~Remark in Section 4 of \cite{ShiodaManifoldsTokyo}):
\begin{equation}\label{eq:thez}
x_1+i\, x_3- x_4=x_2+x_3 -i \,x_4=0.
\end{equation}
Let $\MW(\sigma_F, z_F)$ denote  the Mordell-Weil group of $\sigma_F\colon F_3\to \P^1$
with the zero section $z_F$, and 
let $\Triv (\sigma_F, z_F)$ be the sublattice of $\NS{3}$ generated by the classes of  the zero section $z_F$
and the $24$ lines 
in the singular fibers of $\sigma_F$.
(This lattice is called the \emph{trivial sublattice} of the Jacobian fibration 
$(\sigma_F, z_F)$ in the theory of Mordell-Weil lattices~\cite{ShiodaMWL}.)
Let $\Triv^{-} (\sigma_F, z_F)$ denote the primitive closure
of $\Triv (\sigma_F, z_F)$ in $\NS{3}$.
By~\cite{ShiodaMWL}, we have a canonical isomorphism
\begin{equation}\label{eq:MWtor}
\Triv^{-} (\sigma_F, z_F)/\Triv (\sigma_F, z_F)\;\cong\; 
\textrm{the torsion part of $\MW(\sigma_F, z_F)$.}
\end{equation}
Therefore a section $s\colon \P^1\to F_3$ of $\sigma_F$ is 
a torsion element  of $\MW(\sigma_F, z_F)$
if the class of  $s$ belongs to $\Triv^{-} (\sigma_F, z_F)$.
By this criterion,
we find  $16$ lines among  the $64$ sections of $\sigma_F$
that form the torsion part of $\MW(\sigma_F, z_F)$.
Thus we obtain the configuration $\LLL_{40, 3}$ on  $X_3$ as a sub-configuration
of $\LLL_{112}$.
Combining this embedding $\LLL_{40, 3}\inj \LLL_{112}$ with the bijection  
 $\LLL_{40}=\LLL_{40, 0}\cong \LLL_{40, 3}$
induced by specialization of $X_0$ to $X_3$, 
we obtain the  embedding $\rho_{\LLL}\colon \LLL_{40}\inj \LLL_{112}$
induced by the specialization of $X_0$ to $X_3$.
\par
The dual graph of $\LLL_{40}$
is now calculated explicitly.  
Hence we can prove the following by a direct computation.
\begin{proposition}\label{prop:LLL40QPG1}
The dual graph of $\LLL_{40}$ is isomorphic to the $\QP$-graph $\QPG_1$.
\qed
\end{proposition}
Comparing the ranks and the discriminants of $\latgraph{\LLL_{40}}\cong \latgraph{\QPG_1}$ and $\NS{0}$,
we obtain the following:
\begin{corollary}\label{cor:S0LLL40}
The lattice  $\NS{0}$ is generated by the classes of curves in $\LLL_{40}$. 
\qed
\end{corollary}
\begin{corollary}\label{cor:rhoS0S3}
The embedding $\rho_{\LLL}\colon \LLL_{40}\inj \LLL_{112}$
induces 
the embedding $\rho\colon \NS{0}\inj \NS{3}$
induced  by the specialization  of $X_0$ to $X_3$.
This embedding $\rho$ is primitive.
\qed
\end{corollary}
The last assertion follows from the explicit matrix form of the embedding $\rho$
with respect to some bases of $S_0$ and $S_3$ (see~\cite{thecompdata}).
\begin{remark}
The existence of an isomorphism  $X_3\cong F_3$ can be easily seen 
by the following argument.
By~\cite{ShimadaReduction},
we know that $X_3$ is a supersingular $K3$ surface
with Artin invariant $1$,
and hence is isomorphic to $F_3$
by  the uniqueness of a supersingular $K3$ surface
with Artin invariant $1$. 
\end{remark}
%
%\begin{remark}\label{rem:X4422A}
%The extremal rational elliptic surface $X_{[4,4,2,2]}\to \P^1$ in characteristic $3$
%is defined by $y^2=x(x+1)(x+t^2)$ (see Lang~\cite{Lang1991}).
%Hence $\sigma\colon X_3\to \P^1$ is  obtained from $X_{[4,4,2,2]}\to \P^1$
%by a base-change of degree $2$.
%The induced double covering $X_3\to X_{[4,4,2,2]}$ is branched along the two fibers of type $I_2$
%of $X_{[4,4,2,2]}\to \P^1$, 
%and hence the involution of $X_3$ associated with $X_3\to X_{[4,4,2,2]}$
%must fix $4$ smooth rational curves on $X_3$ pointwisely.
%Unfortunately, there exist no involutions in $\PGU(4, \F_9)$
%that fix $4$ lines pointwisely.
%See Remark~\ref{rem:X4422B} on the involution associated with $X_3\to X_{[4,4,2,2]}$.
%\end{remark}
%
\subsection{All embeddings of $ \LLL_{40}$ into $\LLL_{112}$}
The embedding $\rho_{\LLL}\colon \LLL_{40}\inj \LLL_{112}$ constructed 
in the preceding section  depends 
on the choice of $\sigma_F$ and  $z_F$.
In this section, 
we make the complete list of all embeddings $\LLL_{40}\inj \LLL_{112}$.
\par
Let $a\mapsto \bar{a}:=a^3$ denote the Frobenius automorphism of 
the base field $k_3$. 
Then the  projective automorphism group of $F_3\subset \P^3$
is equal to 
$$
\thePGU:=\set{g\in \GL_4(k_3)}{\textrm{${}^{T} g \cdot \bar{g}$ is a scalar matrix}  }/ k_3\sptimes, 
$$
which is of order $13063680$.
We can calculate  the action of $\thePGU$ on $\LLL_{112}$ and on $\NS{3}=\gen{\LLL_{112}}$.
Let $\AAA$ denote the set of all ordered $5$-tuples
$[z, \ell_0, \dots, \ell_3]$ of lines on $F_3$
that form the configuration whose dual graph is as follows.
\begin{equation}\label{eq:alpha}
\raise -.8cm \hbox{
\setlength\unitlength{.6truecm}
\begin{picture}(7,2.9)(0,.6)
\put(1.4,2){\circle{.3}} 
\put(1.3,2.4){$z$}
\put(3,2){\circle{.3}} 
\put(2.6,2.4){$\ell_0$}
\put(5,2){\circle{.3}} 
\put(5.25,2){$\ell_2$}
\put(4,1){\circle{.3}} 
\put(4.27,.8){$\ell_1$}
\put(4,3){\circle{.3}} 
\put(4.27,3){$\ell_3$}
\put(1.55,2){\line(1,0){1.3}}
\put(3.11,2.12){\line(1,1){0.78}}
\put(3.11,1.88){\line(1,-1){0.78}}
\put(4.12,2.88){\line(1,-1){0.78}}
\put(4.12,1.12){\line(1,1){0.78}}
\end{picture}
}
\end{equation}
Note that $\thePGU$ acts on  $\AAA$ naturally.
We have  the following:
\begin{proposition}\label{prop:AAA}
The action of $\thePGU$ on $\AAA$ is simply transitive.
\end{proposition}
\begin{proof}
By~\cite{ShimadaFermat}, 
we have the following facts.
\begin{enumerate}[(1)]
\item 
Since every line  on $F_3$ is defined over $\F_9$,
the intersection points of 
 $\ell\in \LLL_{112}$ with other lines in $\LLL_{112}$  are  $\F_9$-rational.
For each $\F_9$-rational point $P$ of $\ell$,
there exist exactly three lines in $\LLL_{112} \setminus\{\ell\}$ that intersect $\ell$ at $P$.
Hence there exist exactly $112-3\times 10-1=81$ lines in $\LLL_{112}$ that
are disjoint from $\ell$.
The group $\thePGU$ acts on the set of ordered pairs of
disjoint lines in $\LLL_{112}$.
\item If $\ell_1, \ell_2, \ell_3\in \LLL_{112}$ satisfy
$\intf{\ell_1, \ell_2}=\intf{\ell_2, \ell_3}=\intf{\ell_3, \ell_1}=1$,
then there  exist a plane $\Pi\subset \P^3$ containing  $\ell_1, \ell_2, \ell_3$
and a point $P\in \Pi$ contained in $\ell_1, \ell_2, \ell_3$.
The residual line $\ell_4=(F_3\cap \Pi)-(\ell_1+ \ell_2+ \ell_3)$ also passes through $P$.
\item Let $[\ell_1, \ell_2]$ be an ordered pair of disjoint lines in $\LLL_{112}$.
Then there exist exactly $10$ lines
that intersect both $\ell_1$ and $\ell_2$.
Let $\Stab([\ell_1, \ell_2])$ denote the stabilizer subgroup of $[\ell_1, \ell_2]$
in $\thePGU$.
Then the restriction homomorphism
\begin{equation*}\label{eq:StabKer}
\res_{\ell}\colon \Stab([\ell_1, \ell_2])\to \PGL(\ell_1, \F_9)
\end{equation*}
to the group of linear automorphisms of $\ell_1\cong \P^1$ over $\F_9$ is surjective,
and its kernel is of order $2$.
Let $P$ be an $\F_9$-rational point of $\ell_1$,
and let $m_P, m\sprime_P\in \LLL_{112}$ be the lines that intersect $\ell_1$ at $P$ 
but are disjoint from $\ell_2$.
Then the nontrivial element of  $\Ker (\res_{\ell})$  
exchanges $m_P$ and $m\sprime_P$.
\end{enumerate}
The transitivity of the action of $\thePGU$ on $\AAA$  follows from 
these facts.
Moreover we have
$$
|\AAA|=112\cdot 81\cdot 10\cdot 9\cdot 16=13063680=|\thePGU|,
$$
where the factor $112$ is the number of choices of $\ell_0$ in 
$[z, \ell_0, \dots,  \ell_3]\in \AAA$,
the factor $81$ is the number of choices of $\ell_2$
when $\ell_0$ is given,
the factor $10\cdot 9$ is the number of choices of $\ell_1$ and $\ell_3$
when $\ell_0$ and $\ell_2$ are given,
and the factor $16$ is the number of choices of $z$
for a given quadrangle $[\ell_0, \dots, \ell_3]$.
Therefore the action of $\thePGU$ on $\AAA$  is simply transitive.
\end{proof}
Let $\FFF$ denote the set of sub-configurations of $\LLL_{112}$ isomorphic to $\LLL_{40}$.
Let $\alpha=[z_{\alpha}, \ell_0, \dots, \ell_3]$ be an element of $\AAA$.
Then there exists a unique Jacobian fibration
$$
\sigma_\alpha\colon F_3\to \P^1
$$
with the zero-section $z_{\alpha}$ such that 
$\ell_0+\ell_1+\ell_2+\ell_3$ is a singular fiber of $\sigma_\alpha$.
The Jacobian fibration $(\sigma_F, z_F)$ 
that was used in the construction of  $\rho_{\LLL}$
is obtained as one of the $(\sigma_{\alpha}, z_{\alpha})$.
By Proposition~\ref{prop:AAA},  all Jacobian fibrations
$(\sigma_{\alpha}, z_{\alpha})$ are conjugate under the action of $\thePGU$.
Therefore $(\sigma_{\alpha}, z_{\alpha})$ yields a sub-configuration 
$\LLL_{\alpha}$ of $\LLL_{112}$ isomorphic to $\LLL_{40}$, and 
the map $\alpha\mapsto \LLL_{\alpha}$
gives a surjection $\lambda\colon \AAA\to \FFF$ 
compatible with the action of $\thePGU$.
The size of a fiber of $\lambda$ over $\LLL\sprime\in \FFF$ is
$$
30\times 2\times 16=960,
$$
where the factor $30$ is the number of quadrangles in $\LLL\sprime\cong \LLL_{40}$,
the factor $2$ counts the flipping $\ell_1\leftrightarrow \ell_3$,
and the factor $16$ is the number of choices of the zero-section $z_{\alpha}$.
Thus we obtain the following:
\begin{corollary}\label{cor:13608}
The number of sub-configurations of $\LLL_{112}$
isomorphic to $\LLL_{40}$ is $|\thePGU|/960=13608$, and 
 $\thePGU$ acts on the set of these sub-configurations transitively.
\qed
\end{corollary}
\subsection{An elliptic modular surface of level $4$ over a discrete valuation ring}
\label{subsec:geomQPcovering}
Let $R$ be a  discrete valuation ring such that $2\in R\sptimes$ and 
$i=\sqrt{-1}\in R$.
We construct a model of the elliptic modular surface of level $4$ over $R$, 
that is, we perform over $R$  the resolution of the completion of the affine surface defined by~\eqref{eq:theellfib4}.
This construction explains the isomorphism $\LLL_{40}\cong \QPG_1$ of graphs  geometrically.
\par
In this paragraph, all schemes and morphisms are defined  over $R$.
We consider the complete quadrangle on $\P^2$ (Figure~\ref{fig:completequadrangle})
such that each of the triple points $t_1, \dots, t_4$ is an $R$-valued point. 
\begin{figure}
\begin{center}
%\hskip -2cm 
\includegraphics{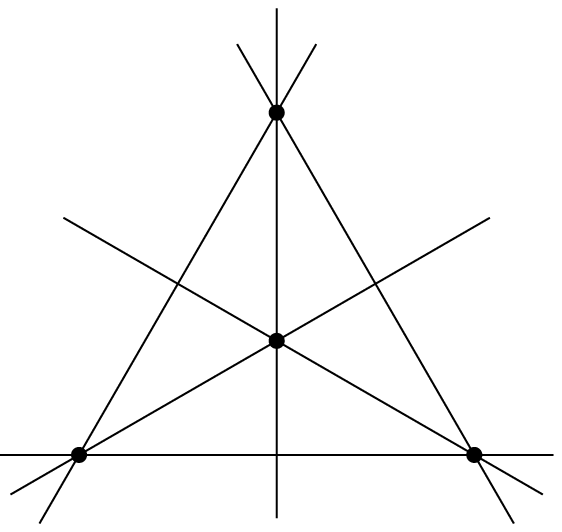}%
\input{completequadrangle.tex}
\end{center}
\caption{Complete quadrangle}\label{fig:completequadrangle}
\end{figure}
Let $M\to \P^2$ be the blow-up of $\P^2$ at $t_1, \dots, t_4$.
Let $\bar{l}_1, \dots, \bar{l}_6$ be 
the strict transforms of the lines $l_1, \dots, l_6$,
and let $\bar{t}_1, \dots, \bar{t}_4$ be 
the exceptional divisors over  $t_1, \dots, t_4$.
It is well-known that these $6+4=10$ smooth rational curves on $M$
form a configuration whose dual graph is the Petersen graph $\PG$.
Let 
\begin{equation}\label{eq:varphiM}
\varphi_M\colon M\to \P^1
\end{equation}
be the fibration induced by
the pencil of lines on $\P^2$ passing through  $t_1$.
(The dependence of the construction on the choice of this $\P^1$-fibration $\varphi_M$ 
will be discussed in Section~\ref{subsec:AutX0h0}.
See Remark~\ref{rem:5fs}.)
Then $\varphi_M$ has exactly three singular fibers 
$\bar{l}_1+\bar{t}_4$, $\bar{l}_2+\bar{t}_3$, $\bar{l}_3+\bar{t}_2$,
and four sections $\bar{t}_1,\bar{l}_4, \bar{l}_5, \bar{l}_6$.
Let $M\sprime\to M$ be the blow-up at the nodes on
$\bar{l}_1+\bar{t}_4$, $\bar{l}_2+\bar{t}_3$, $\bar{l}_3+\bar{t}_2$,
and let $\varphi\sprime_M\colon M\sprime\to \P^1$ 
be the composite of $\varphi_M$ and $M\sprime\to M$.
We choose an affine parameter $\lambda$ on 
the base curve $\P^1$ of $\varphi\sprime_M$ such that
the singular fibers are located over $\lambda=0, 1, \infty$.
Let $ \tilde{M}\sprime\to \P^1$  
be the pull-back of $\varphi_M\sprime\colon M\sprime\to \P^1$  by 
the covering $\P^1\to \P^1$ given by
\begin{equation}\label{eq:lambdasigma}
\sigma\mapsto \lambda=((\sigma+\sigma\inv)/2)^2,
\end{equation}
and let  $\tilde{M}\to \tilde{M}\sprime$ be the normalization of $\tilde{M}\sprime$.
Then $\tilde{M}$ is smooth over $R$, 
and the natural morphism $\tilde\varphi_M\colon \tilde{M}\to \P^1$ to the
$\sigma$-line has exactly  $6$ singular fibers 
over $\sigma=0, \pm 1, \pm i, \infty$.
Each singular fiber is a union of three smooth rational curves
forming the configuration 
$
\setlength\unitlength{.2truecm}
\begin{picture}(8,2)(0,0)
\put(1,.5){\circle{1}}
\put(1.5,.5){\line(1,0){2}}
\put(4,.5){\circle{1}}
\put(4.5,.5){\line(1,0){2}}
\put(7,.5){\circle{1}}
\end{picture},
$ 
the middle of which is with multiplicity $2$.
Let $\tilde{t}_1,\tilde{l}_4, \tilde{l}_5, \tilde{l}_6$ be the pull-backs of 
the sections $\bar{t}_1,\bar{l}_4, \bar{l}_5, \bar{l}_6$ of $\varphi_M$
by $\tilde{M}\to M$.
For a divisor $D$ on $\tilde{M}$,
let $[D]$ denote the class of $D$ in the Picard group
$\Pic\, \tilde{M}$.
Note that, via  $\tilde{M}\to M$, a fiber $F$ of $\varphi_M\colon M\to \P^1$
is pulled back to a sum of two fibers of  $\tilde\varphi_M\colon \tilde{M}\to \P^1$,
and hence the class $[\tilde{F}]$ of the pull-back $\tilde{F}$ of $F$ is 
divisible by $2$ in  $\Pic\, \tilde{M}$.
Let $[H]\in \Pic\, \tilde{M}$ denote the class of
the pull-back of a general line of $\P^2$.
We put $B:=\tilde{t}_1+\tilde{l}_4+\tilde{l}_5+\tilde{l}_6$.
Since $[\tilde{F}]=[H]-[\tilde{t}_1]$ and 
$[\tilde{l}_i]=[H]-[\tilde{t}_j]-[\tilde{t}_{k}]$
for $(i, j, k)=(4,3,4), (5,2,3), (6, 2, 4)$, 
we have 
\[
[B]=3[\tilde{F}]+2[2 \tilde{t}_1-\tilde{t}_2-\tilde{t}_3-\tilde{t}_4].
\]
Therefore $[B]$
is divisible by $2$ in  $\Pic\, \tilde{M}$, 
and we can construct 
a double covering  $\XXX\to \tilde{M}$  branched along $B$.
Then  $\XXX$ is a model of the elliptic modular surface of level $4$ over $R$,
and the Jacobian fibration $\sigma\colon \XXX\to \P^1$ is obtained 
as the composite of the double covering $\XXX\to \tilde{M}$ and  $\tilde\varphi_M\colon \tilde{M}\to \P^1$.
\par
\medskip
The $\QP$-covering map $\LLL_{40}\to \PG$  (see Corollary~\ref{cor:uniqueQPgamma}) is constructed as follows.
We consider an $F$-valued point of $\Spec R$,
where $F$ is a field.
We put $X_F:=\XXX\tensor _R F$, and $\tilde{M}_F:=\tilde{M}\tensor _R F$, $M_F:=M\tensor_R F$.
Let $\EEE_F$ be the generic fiber of $\sigma\tensor F\colon X_F\to \P^1_F$,
which is an elliptic curve over the function field $F(\sigma)$ defined by~\eqref{eq:theellfib4}.
Let $m_2\colon X_F\to X_F$ be the rational map 
induced by multiplication by $2$ 
on $\EEE_F$.
Then  the rational map
\begin{equation}\label{eq:muF}
\mu_F\;\;\colon\;\; X_F \maprightsp{m_2} X_F\; \maprightsp{} 
\tilde{M}_F \maprightsp{} M_F 
\end{equation}
gives a map from $\LLL_{40}$ to the Petersen graph $\PG$
formed by $\{\bar{t}_1, \dots, \bar{t}_4, \bar{\ell}_1, \dots, \bar{\ell}_6\}$.
\begin{proposition}\label{prop:Galois}
The rational map $\mu_F$ induces a Galois extension of the function fields.
Its Galois group $\Gal(\mu)$ is isomorphic to $(\Z/2\Z)^5$ and is 
generated by the inversion $\iota\colon (X, Y, \sigma)\mapsto (X, -Y, \sigma)$  
of the elliptic curve $\EEE_F$, two  involutions 
\begin{equation}\label{eq:2invols}
(X, Y, \sigma)\mapsto (X, Y, -\sigma), \;\;
(X, Y, \sigma)\mapsto (X, Y, 1/\sigma),
\end{equation}
and  the translations by  the $2$-torsion points of  $\EEE_F$.
\end{proposition}
\begin{proof}
The inversion $\iota$ and the involutions in~\eqref{eq:2invols}
fix  each $2$-torsion point of $\EEE_F$.
Hence the  involutions in the statement of Proposition~\ref{prop:Galois} 
generate a group isomorphic to $(\Z/2\Z)^5$.
By~\eqref{eq:lambdasigma},
the function field $F(\sigma)$ is a Galois extension of $F(\lambda)$
with Galois group  generated by $\sigma\mapsto -\sigma$ and $\sigma\mapsto1/\sigma$.
Hence 
the covering $\tilde{M}_F \to M_F$ in~\eqref{eq:muF} is the quotient by the
involutions in~\eqref{eq:2invols}.
The covering $X_F\to \tilde{M}_F $ in~\eqref{eq:muF} is the quotient by $\iota$,
and the map $m_2$ is  the quotient by
the group of  translations by  the $2$-torsion points of  $\EEE_F$.
Thus the proof is completed.
\end{proof}
\subsection{Another model of the elliptic modular surface of level $4$}
\label{subsec:anothermodel}
We give a much simpler construction of a $(\Z/2\Z)^5$-covering  $X_{0}\to M_{\C}$
 over the complex numbers
by means of a Hirzebruch covering~(see Hironaka~\cite{EHironaka1993}).
This section is due to a suggestion by one of the referees of the
first version of the paper.
Let $M_{\C}$ be the complex surface obtained by blowing-up $\P^2_{\C}$ at 
the triple points of the complete quadrangle on $\P^2_{\C}$,
and let $M\sp{\circ}_{\C}$ be the complement of the ten $(-1)$-curves on $M_{\C}$.
We have a canonical surjective homomorphism 
$\pi_1 (M\sp{\circ}_{\C}) \surj H_1(M\sp{\circ}_{\C}, \Z/2\Z)\cong (\Z/2\Z)^5$.
It is known (see~\cite{EHironaka1993})
that the corresponding \'etale covering
$W\sp{\circ}\to M\sp{\circ}_{\C}$ extends to a finite morphism $W\to M_{\C}$
from a smooth surface  $W$, and 
that $W$ is a $K3$ surface.
\begin{proposition}
The surface $W$ has a Jacobian fibration $\sigma_W\colon W\to \P^1$
that is isomorphic to  $\sigma\colon X_0\to \P^1$.
\end{proposition}
\begin{proof}
Consider the $(\Z/2\Z)^5$-covering $\gamma\colon \P^5\to \PPb^5$ defined by 
\[
%[x_0: x_1: x_2: x_3: x_4: x_5: x_6]\mapsto [X_0: X_1: X_2: X_3: X_4: X_5: X_6]=[x_0^2: x_1^2: x_2^2: x_3^2: x_4^2: x_5^2: x_6^2].
[x_0: x_1: \dots : x_6]\mapsto [X_0: X_1:\dots : X_6]=[x_0^2: x_1^2:\dots : x_6^2].
\]
Let $P\subset \PPb^5$ be the linear plane defined by
\begin{equation*}\label{eq:defP}
X_1-X_2+X_3=-X_3+X_5+X_6=X_2+X_4-X_5=0,
\end{equation*}
and, for $i=1, \dots, 6$,
let $l_i\subset P$ denote the intersection of $P$ and the coordinate hyperplane $X_i=0$.
Then the $6$ lines $l_1, \dots, l_6$ form the complete quadrangle in Figure~\ref{fig:completequadrangle}.
The surface $\overline{W}:=\gamma\inv (P)\subset \P^5$
is the complete intersection of   three quadratic hypersurfaces
\begin{equation}\label{eq:def3Qs}
x_1^2-x_2^2+x_3^2=-x_3^2+x_5^2+x_6^2=x_2^2+x_4^2-x_5^2=0.
\end{equation}
The finite covering  $\gamma|\overline{W}\colon \overline{W}\to P$
extends to the covering $\gamma_W\colon W\to M_{\C}$ by the blowing up  of $M_{\C}\to P$
at the triple points $t_1, \dots, t_4$ of 
the complete quadrangle on $P$.
The pull-back of each line $l_i$ by $\gamma|\overline{W}$ is a union of $4$ conics, and 
$\overline{W}$ has $4\times 4$ nodes over $t_1, \dots, t_4$.
%By the resolution $W\to \overline{W}$,
Thus we obtain a configuration $\LLL_{W}$ of $40$ smooth rational curves on $W$
consisting of $4\times 6$ pullbacks of  conics on $\overline{W}$ and $4\times 4$
exceptional curves over the nodes of $\overline{W}$.
By computing 
the intersection numbers of the $24$ conics and 
the incidence relation between the conics and the $16$ nodes,
we can write the intersection matrix of the  configuration $\LLL_{W}$
explicitly.
Then we confirm that this configuration $\LLL_{W}$ is isomorphic to $\LLL_{40}$.
In fact, by Proposition~\ref{prop:AutQPG},
there exist $7680$ isomorphisms between $\LLL_{W}$ and $\LLL_{40}$.
Among these isomorphisms, we have $1536$ isomorphisms 
such that the $16$ smooth rational curves corresponding to the nodes
of $\overline{W}$ are mapped to  the sections of $\sigma\colon X_0\to \P^1$
and the $24$ smooth rational curves over the lines $l_i$
are mapped to the irreducible components of singular fibers of $\sigma$.
Hence $W$ has an elliptic fibration  $\sigma_W\colon W\to \P^1$ with a section
and  $6$ singular fibers of type $I_4$.
By~\cite{ShimadaConn},
such an elliptic $K3$ surface is unique up to isomorphism.
Hence   $\sigma_W\colon W\to \P^1$ 
is isomorphic to $\sigma\colon X_0\to \P^1$.
\end{proof}
\begin{remark}\label{rem:sigmaW}
The Jacobian fibration $\sigma_W\colon W\to \P^1$ is obtained
from the elliptic fibration $M_{\C}\to \P^1$
induced by the pencil of conics passing through all the triple points $t_1, \dots, t_4$. 
See Remark~\ref{rem:5fs},
which also explains the number $1536=7680/5$ of the special isomorphisms 
$\LLL_{W}\cong \LLL_{40}$
in the proof.
\end{remark}
For  $J\subset \{1, \dots, 6\}$,
let  $\tilde{\tau}_J$ denote the 
involution of $\P^5$ given by
\[
x_m\mapsto -x_m \;\; \textrm{if}\;\; m \in J,
\quad
x_n\mapsto x_n \;\; \textrm{if}\;\; n \notin J.
\]
Note that $\tilde{\tau}_{J}=\tilde{\tau}_{J\sprime}$ if 
$J\cap J\sprime=\emptyset$ and $J\cup J\sprime=\{1, \dots, 6\}$.
The Galois group $\Gal(\gamma_W)$ 
of the covering  $\gamma_W\colon  W\to M_{\C}$
consists of the restrictions $\tau_J:=\tilde{\tau}_J|\overline{W}$
of these involutions $\tilde{\tau}_J$ to $\overline{W}$.
Let $S_W$ denote the N\'eron-Severi lattice of $W$,
which is equal to $\gen{\LLL_W}$.
We can calculate the action of  $\Gal(\gamma_W)$
on $S_W$ explicitly.
\par
For an isomorphism $\varphi\colon \LLL_W\cong \LLL_{40}$
of graphs,
let $\gen{\varphi}\colon S_W\cong S_0$
denote the induced isometry of lattices,
and let $\OG(\gen{\varphi})\colon \OG(S_W)\cong \OG(S_0)$
denote the induced isomorphism  of the automorphism groups of lattices.
By checking  
all the $7680$ isomorphisms $\varphi\colon \LLL_W\cong \LLL_{40}$,
we confirmed the following fact.
See Remark~\ref{rem:5fs} for a geometric reason of this result.
\begin{proposition}\label{prop:GalGal}
For each isomorphism $\varphi\colon \LLL_W\cong \LLL_{40}$
of graphs,  
the isomorphism $\OG(\gen{\varphi})$ maps $\Gal(\gamma_W)\subset\OGplus(S_W)$
to $\Gal(\mu)\subset \OGplus(S_0)$ isomorphically.
\qed
\end{proposition}
By Barth--Hulek~\cite{BH1985},
we know that 
the sum $I$ of the classes of sections of $\sigma\colon X_0\to \P^1$
is divisible by $2$ in $\Pic\, X_0$.
We put  $h_8:=(1/2)I+F$, where $F\in \Pic\, X_0$ is a fiber of $\sigma$.
Then $h_8$ is primitive in $\Pic\, X_0$ and  nef of degree $8$.
The complete linear system $|h_8|$ is base point free, 
because  there exist no vectors $f\in S_0$
such that $\intf{f, f}=0$ and $\intf{f, h_8}=1$
(see Nikulin~\cite{Nikulin91} and Proposition~12 of~\cite{BH1985}).
Let $\Phi_{8}\colon X_0\to\P^5$ be the morphism induced by $|h_8|$.
The curves contracted by $\Phi_8$
are exactly the sections of $\sigma\colon X_0\to \P^1$, and 
$\Phi_8$ maps each irreducible component of singular fibers of $\sigma$
to a conic.  
Hence the image of  $\Phi_8$
is equal to $\overline{W}$.
We consider the involutions $\tau_J$ of $\overline{W}$
as  elements of $\Aut(X_0)$ via the birational morphism $\Phi_8$.
By Proposition~\ref{prop:GalGal},
we have the following description of 
$\Gal(\mu)$ simpler than Proposition~\ref{prop:Galois}.
\begin{proposition}\label{prop:Galoct}
The Galois group  $\Gal(\mu)$ 
consists of $32$ involutions $\tau_J$.
\qed
\end{proposition}
%
%\begin{corollary}
%The rational map $\mu_{\C}\colon X_0 \maprightsp{} M_{\C}$ is in fact a morphism.
%\qed
%\end{corollary}
%
\begin{remark}
In~\cite{AST1996},
Abo--Sasakura--Terasoma studied $X_p$, where $p\equiv 1 \bmod 4$,
and obtained an isomorphism from $X_p$
to the reduction of the complete intersection~\eqref{eq:def3Qs}
modulo $p$.
\end{remark}
\section{Borcherds' method}\label{sec:Borcherds}
\subsection{Chambers}\label{subsec:chamber}
We fix notions about tessellation 
of a positive cone of an even hyperbolic lattice 
by chambers.
\par
Let $L$ be an even   lattice.
A vector $r\in L$ is called a \emph{root} if $\intf{r,r}=-2$.
The set of roots of $L$ is denoted by $\Roots(L)$.
%We say that $L$ is a \emph{root lattice} if $L$ is generated by $\Roots(L)$.
\par
Let $L$ be an even  hyperbolic lattice.
Let $\PPP (L)$ be one of the two connected components of $\shortset{x\in L\tensor\R}{\intf{x, x}>0}$.
Then  $\OG^+(L)$ acts on  $\PPP (L)$.
For  $v\in L\tensor \Q$ with $\intf{v,v}<0$,
let $(v)\sperp$ denote the hyperplane of $\PPP(L)$
defined by $\intf{x, v}=0$.
Let $\VVV$ be a set of vectors of $L\tensor\Q$
such that $\intf{v, v}<0$   for all $v\in \VVV$.
We assume  that \emph{the family $\shortset{(v)\sperp}{v\in \VVV}$ of hyperplanes
is locally finite in  $\PPP(L)$}.
A \emph{$\VVV$-chamber} is the closure in $\PPP(L)$ 
of a connected component of
$$
\PPP(L)\;\; \setminus \;\; \bigcup_{v\in \VVV} (v)\sperp.
$$
Typical examples are $\Roots(L)$-chambers defined by the set $\Roots(L)$ of roots of $L$.
\begin{definition}
Let $N$ be a closed subset of $\PPP(L)$.
We say that \emph{$N$ is tessellated by $\VVV$-chambers} if 
$N$ is  a union of $\VVV$-chambers.
Suppose that $N$ is  tessellated by $\VVV$-chambers, and 
let $H$ be a subgroup of $\OGplus(L)$ that preserves $N$.
We say that \emph{$H$ preserves the  tessellation of $N$ by $\VVV$-chambers}
if any $g\in H$ maps each $\VVV$-chamber  in $N$ to a $\VVV$-chamber.
Suppose that  this is the case.
We say that the tessellation of $N$  is \emph{$H$-transitive}
if $H$ acts transitively on the set of  $\VVV$-chambers in $N$.
\end{definition}
\begin{remark}\label{rem:refinement}
Let $U$ be a subset of $\VVV$ such that the closed subset 
$$
N_U:=\set{x\in \PPP(L)}{\intf{x, v}\ge 0\;\;\textrm{for all $v\in U$}}
$$
of $\PPP(L)$ contains an interior point.
Then $N_U$  is tessellated by $\VVV$-chambers.
In particular, if $\VVV\sprime$ is a subset of $\VVV$,
then each $\VVV\sprime$-chamber is tessellated by $\VVV$-chambers.
\end{remark}
Let $D$ be a $\VVV$-chamber.
We put
$$
\Aut (D):=\set{g\in \OG^+(L)}{D^g=D}.
$$
A \emph{wall  of $D$} is a closed subset of $D$ of the form $(v)\sperp\cap D$
such that 
the hyperplane $(v)\sperp$ of $\PPP(L)$ is disjoint from the interior of $D$ and 
$(v)\sperp\cap D$ contains a non-empty open subset of $(v)\sperp$.
We say that a hyperplane $(v)\sperp$ of $\PPP(L)$ \emph{defines a wall of $D$} 
if $(v)\sperp\cap D$ is a wall of $D$.
We say that a vector $v\in L\tensor\Q$ with $\intf{v,v}<0$
\emph{defines a wall of $D$} if $(v)\sperp$ defines a wall of $D$
and $\intf{v, x}\ge 0$ for all $x\in D$.
Note that, 
for each wall of $D$,
there exists a unique \emph{primitive} vector in $L\dual$ defining the wall.
Let $(v)\sperp\cap D$ be a wall of $D$.
Then there exists a unique $\VVV$-chamber $D\sprime$ such that 
the interiors of $D$ and $D\sprime$ are disjoint and 
that $(v)\sperp \cap D$ is equal to $(v)\sperp \cap D\sprime$.
(Hence  $(v)\sperp \cap D\sprime$ is a wall of $D\sprime$.)
We say that $D\sprime$ is a $\VVV$-chamber \emph{adjacent to $D$ across the wall $(v)\sperp\cap D$}.
A \emph{face} of $D$ is a closed subset of $D$ 
of the form $F\cap D$ such that  
$$
F=(v_1)\sperp\cap \dots\cap (v_m)\sperp, \quad \textrm{where $(v_1)\sperp,  \dots,  (v_m)\sperp$ define walls of $D$, }
$$
and  that $F\cap D$ contains a non-empty open subset of $F$.
\begin{example}
We consider 
the tessellation of $\PPP(L)$ by $\Roots(L)$-chambers.
Each root $r$ of $L$ defines a \emph{reflection} $s_r\in \OG^+(L)$ via
$x\mapsto x+\intf{x, r}r$.
Let $W(L)$ denote the subgroup of $\OG^+(L)$ generated by 
all the reflections with respect to the roots.
Then the  tessellation of $\PPP(L)$ by $\Roots(L)$-chambers is $W(L)$-transitive.
An $\Roots(L)$-chamber $N$ is a fundamental domain of the action of $W(L)$ on $\PPP(L)$, and 
$\OG^+(L)$ is equal to $ W(L)\semidirectproduct\Aut(N)$.
Moreover, 
$W(L)$ is generated by the reflections $s_r$
associated with the roots $r$ of $L$ defining the walls of $N$,
and the faces of codimension $2$ of $N$ give the defining relations of $W(L)$
with respect to this set of generators.
\end{example}
Let $L_{26}$ be an even \emph{unimodular} hyperbolic lattice of rank $26$,
which is unique up to isomorphism.
The shape of an $\Roots(L_{26})$-chamber was determined by Conway~\cite{Conway1983},
and hence we call an $\Roots(L_{26})$-chamber  a \emph{Conway chamber}.
Let $w$ be 
a non-zero primitive vector  of $L_{26}$ with $\intf{w, w}=0$ 
such that $w$ is  contained in the closure of $\PPP(L_{26})$ in $L_{26}\tensor\R$.
We say that $w$ is a \emph{Weyl vector}
if  the lattice $\gen{w}\sperp/\gen{w}$
is isomorphic to the negative-definite Leech lattice,
where $\gen{w}\sperp$ is the orthogonal complement in $L_{26}$ of
 $\gen{w}:=\Z w\subset L_{26}$.
Let $w\in L_{26}$ be a Weyl vector.
Then a root $r$ of $L_{26}$ is called a \emph{Leech root with respect to $w$} 
if $\intf{w, r}=1$.
We put
$$
\CCC(w):=\set{x\in \PPP(L_{26})}{\intf{x, r}\ge 0\;\;\textrm{for all Leech roots $r$ with respect to $w$}}.
$$
\begin{theorem}[Conway~\cite{Conway1983}]
The mapping $w\mapsto \CCC(w)$
gives a bijection from the set of Weyl vectors to the set of Conway chambers.
\end{theorem}
\subsection{Borcherds' method}\label{subsec:Borcherds}
Borcherds~\cite{Bor1, Bor2}  developed a method to analyze $\Roots(S)$-chambers of an even hyperbolic lattice $S$
by means of Conway chambers.
We briefly review this method, and fix some terminologies.
See~\cite{ShimadaAlgo} for details of the algorithms. 
\par
Let $S$ be an even hyperbolic lattice.
Suppose that we have  a primitive embedding
$i\colon S\inj L_{26}$
such that the orthogonal complement $R$ of $S$ in $L_{26}$
satisfies the following condition:
\begin{equation}\label{eq:Rcond}
\textrm{$R$ cannot be embedded into the negative-definite Leech lattice.}
\end{equation}
(This condition is fulfilled, for example,  if $R$ contains a root.)
We choose $\PPP(S)$ so that the embedding $i\colon S\inj L_{26}$ induces 
an embedding $i_{\PPP}\colon \PPP(S)\inj \PPP(L_{26})$.
Let
$$
\pr_S\colon L_{26}\tensor\Q\to S\tensor\Q
$$
denote the orthogonal projection.
A hyperplane $(v)\sperp$ of $\PPP(L_{26})$
intersects $\PPP(S)$ in a hyperplane  if and only if $\intf{\pr_S(v), \pr_S(v)}<0$,
and,  if this is the case, we have $\PPP(S)\cap (v)\sperp=(\pr_S(v))\sperp$.
We put
\begin{equation}\label{eq:VVVS}
\prV  (i):=\set{\pr_S(r)}{r\in \Roots(L_{26}), \;\; \intf{\pr_S(r), \pr_S(r)}<0}.
\end{equation}
The tessellation of $\PPP(L_{26})$ by Conway chambers
induces a tessellation of $\PPP(S)$ by  $\prV(i)$-chambers.
Each $\prV(i)$-chamber is of the form $i_{\PPP}\inv (\CCC(w))$.
It is easily seen (see~\cite{ShimadaAlgo}) that the assumption~\eqref{eq:Rcond} implies that 
 each $\prV (i)$-chamber has only a finite number of walls.
The defining vectors  of walls of a $\prV (i)$-chamber $i_{\PPP}\inv (\CCC(w))$
can be calculated from the Weyl vector $w\in L_{26}$ of the Conway chamber $\CCC(w)$.
From this set of walls of  $i_{\PPP}\inv (\CCC(w))$,
 we can  calculate the finite group $\Aut(i_{\PPP}\inv (\CCC(w)))\subset \OGplus(S)$.
Moreover, for each wall $(v)\sperp\cap i_{\PPP}\inv (\CCC(w))$ of a $\prV (i)$-chamber $i_{\PPP}\inv (\CCC(w))$,
we can calculate a Weyl vector $w\sprime$
such that $i_{\PPP}\inv (\CCC(w\sprime))$ is the $\prV (i)$-chamber
adjacent to $i_{\PPP}\inv (\CCC(w))$ across the wall $(v)\sperp\cap i_{\PPP}\inv (\CCC(w))$.
\par
Since   $\Roots(S)\subset  \prV (i)$,
 Remark~\ref{rem:refinement}
implies  the following:
\begin{proposition}\label{prop:RSVS}
An $\Roots(S)$-chamber is tessellated by  $\prV (i)$-chambers.
\qed
\end{proposition}
%
%Using this tessellation by  $\prV (i)$-chambers,
%we obtain a detailed description of an $\Roots(S)$-chamber.
%
\subsection{Discriminant forms}
For the application of Borcherds' method
to $K3$ surfaces,
we need the notion of discriminant forms 
due to Nikulin~\cite{Nikulin79}.
\par 
Let $q\colon A\to \Q/2\Z$ be a nondegenerate quadratic form 
with values in $\Q/2\Z$
on a finite abelian group $A$.
We denote by $\OG(q)$  the automorphism group of  $(A, q)$.
For a prime $p$,
we denote by $A_p$ the $p$-part of $A$ and 
by $q_p\colon A_p\to \Q/2\Z$ the restriction of $q$ to $A_p$.
Then we have a canonical orthogonal direct-sum decomposition
$$
(A, q)=\bigoplus (A_p, q_p).
$$
Hence $\OG(q)$ is canonically isomorphic to the direct product of $\OG(q_p)$.
\par
Let $L$ be an even lattice,
and let $\discg{L}=L\dual/L$ denote the discriminant group of $L$.
We define the  \emph{discriminant form of $L$}
$$
\discf{L}\colon \discg{L}\to \Q/2\Z
$$
by $\discf{L}(\bar x):=\intf{x, x}\bmod 2\Z$,
where $x\mapsto \bar{x}$ is the natural projection $L\dual\to \discg{L}$.
Then we have a natural homomorphism
$$
\eta_L\colon \OG(L)\to \OG(\discf{L}).
$$
\par
Let $M$ be a primitive sublattice of an even lattice $L$,
and  $N$  the orthogonal complement of $M$ in $L$.
Let $\OG(L, M)$ denote  the subgroup
$\shortset{g\in \OG(L)}{M^g=M}$
of $\OG(L)$.
Then we have a canonical embedding
$\OG(L, M)\inj \OG(M)\times \OG(N)$.
The submodule $L\subset M\dual\oplus N\dual$
defines a subgroup $\Gamma_L:=L/(M\oplus N)\subset \discg{M} \times \discg{N}$.
By Nikulin~\cite{Nikulin79}, we have the following:
\begin{proposition}\label{prop:LN}
Let $p$ be a prime that does not divide $|\discg{M}|$.
Then $N\inj L$ induces  an isomorphism
$\discf{L}_p\cong \discf{N}_p$,
which is compatible with the actions of $\OG(L, M)$ on $L$ and  on $N$.
\qed
\end{proposition}
 \begin{proposition}\label{prop:MN}
 Let $p$ be a prime that does not divide $|\discg{L}|$.
Then the $p$-part of $\Gamma_L$ is the graph of  an isomorphism
$\discf{M}_p\cong -\discf{N}_p$,
which is compatible with the actions of $\OG(L, M)$ on $M$ and on $N$.
\qed
 \end{proposition}
 \begin{proposition}\label{prop:unimoduoverlat}
 Suppose that $L$ is unimodular, and 
let $\gamma_L\colon \discf{M}\cong -\discf{N}$ be the isomorphism 
with the graph $\Gamma_L$.
Let  $H$ be a subgroup of $\OG(N)$.
Then $g\in \OG(M)$ extends to $\tilde{g}\in \OG(L, M)$ with $\tilde{g}|_N\in H$
if and only if the isomorphism $\OG( \discf{M})\cong \OG(\discf{N})$
induced by $\gamma_L$ maps 
$\eta_M(g)\in \OG(\discf{M})$ into $\eta_N(H)\subset \OG(\discf{N})$.
\qed
 \end{proposition}
\subsection{Geometric application of Borcherds' method}\label{subsec:geomBorcherds}
Let  $Z$ be a $K3$ surface defined over an algebraically closed field.
We use the notation $\NS{Z}$, $\PP{Z}$ and $\NN{Z}$ defined in Section~\ref{subsec:ShiodaKK}.
The following is well-known.
\begin{proposition}\label{prop:NZNZg}
The closed subset  $\NN{Z}$ of $ \PP{Z}$ is an $\Roots(\NS{Z})$-chamber.
The mapping $C\mapsto ([C])\sperp\cap \NN{Z}$ gives a one-to-one correspondence
between the set of smooth rational curves on $Z$ and the set of walls of $\NN{Z}$.
\qed
\end{proposition}
Since the action of  $\OG^+(\NS{Z})$ on $ \PP{Z}$ preserves the tessellation by $\Roots(\NS{Z})$-chambers
and an ample class is an interior point of $\NN{Z}\subset \PP{Z}$, 
we obtain the following.
\begin{corollary}\label{cor:Npreserve}
Let $a\in \NS{Z}$ be an ample class.
%Since $a$ is an interior point of $\NN{Z}$, 
Then  the following three conditions on $g\in \OG^+(\NS{Z})$ are equivalent:
{\rm (i)} $\NN{Z}=\NN{Z}^g$.
{\rm (ii)} $\NN{Z}\cap \NN{Z}^g$ contains an interior point of $\NN{Z}$.
{\rm (iii)} There exist no roots $r$ of $\NS{Z}$ such that $\intf{r, a}$ and $\intf{r, a^g}$ have different signs.
\qed
\end{corollary}
Let $Z$ be a complex  $K3$ surface.
Let $\TL{Z}$ denote the orthogonal complement of $\NS{Z}=H^2(Z, \Z)\cap H^{1,1}(Z)$ in 
the even unimodular lattice $H^2(Z, \Z)$ with the cup-product.
Then $\TL{Z}\tensor\C$ contains a  one-dimensional subspace $H^{2, 0}(Z)=\C\, \omega$,
where $\omega$ is a non-zero holomorphic $2$-form on $Z$.
We put
$$
\OG(\TL{Z}, \omega):=\set{g\in \OG(\TL{Z})}{\C\, \omega^g=\C\, \omega}.
$$
Recall that we have a natural homomorphism 
$\eta_{\TL{Z}}\colon \OG(\TL{Z})\to \OG(\discf{\TL{Z}})$.
We  put
 $$
 \OG(\discf{\TL{Z}}, \omega):=\textrm{the image of $\OG(\TL{Z}, \omega)$ 
 under $\eta_{\TL{Z}}$.}
$$
The  even unimodular overlattice $H^2(Z, \Z)$ of $\NS{Z}\oplus \TL{Z}$ 
induces  an isomorphism 
$\gamma_H$ between  $\discf{\NS{Z}}$ and $-\discf{\TL{Z}}$.
Let $\OG(\discf{\NS{Z}}, \omega)$ 
denote the subgroup of $\OG(\discf{\NS{Z}})$ corresponding to 
$\OG(\discf{\TL{Z}}, \omega)$ 
via the isomorphism $\OG(\discf{\TL{Z}})\cong \OG(\discf{\NS{Z}})$
induced by  $\gamma_H$.
By Proposition~\ref{prop:unimoduoverlat},
an isometry  $g\in \OG(\NS{Z})$ extends to an isometry $\tilde{g}$ of $H^2(Z, \Z)$
that preserves $H^{2, 0} (Z)$
if and only if
$\eta_{\NS{Z}}(g)\in \OG(\discf{\NS{Z}}, \omega)$.
\par
Let $Z$ be a supersingular $K3$ surface
defined over an algebraically closed field $k_p$ of odd characteristic $p$.
Then $\discg{\NS{Z}}$ is an $\F_p$-vector space,
and we have the \emph{period} of $Z$,
which is a  subspace of $\discg{\NS{Z}}\tensor k_p$.
(See Ogus~\cite{Ogus1, Ogus2}.)
Let $\OG(\discf{\NS{Z}}, \omega)$ denote the subgroup of $\OG(\discf{\NS{Z}})$ consisting 
of automorphisms that preserve the period.
\par
In the two cases where $Z$ is defined over $\C$  or  supersingular  in odd characteristic,
we call the condition
\begin{equation}\label{eq:periodcondition}
\eta_{\NS{Z}}(g)\in \OG(\discf{\NS{Z}}, \omega)
\end{equation}
on $g\in \OGplus(\NS{Z})$ the \emph{period condition}.
In these two cases,
we have the Torelli theorem.
(See Piatetski-Shapiro and Shafarevich~\cite{PSS71}, Ogus~\cite{Ogus1, Ogus2} for $p>3$
and Bragg and Lieblich~\cite{BL2018} for $p\ge 3$.)
By virtue of this theorem, we have  the following:
\begin{theorem}\label{thm:TorelliAut}
Let $Z$ be a complex $K3$ surface or a supersingular $K3$ surface in odd characteristic,
and let 
$\psi_Z\colon \Aut(Z)\to \OGplus(\NS{Z})$
be the natural representation of $\Aut(Z)$ on $\NS{Z}$.
Then an isometry $g\in \OGplus (\NS{Z})$
belongs to the image of $\psi_Z$ if and only if $g$ preserves $\NN{Z}$ and 
satisfies the period condition~\eqref{eq:periodcondition}.
\qed
\end{theorem}
We explain the procedure of Borcherds' method
in the simplest case. 
See~\cite{ShimadaAlgo} for more general cases.
In the following,
we assume that 
$Z$ is a complex $K3$ surface or a supersingular $K3$ surface in odd characteristic.
We also assume that 
 \emph{$\psi_Z$ is injective, and
regard $\Aut(Z)$ as a subgroup of $\OGplus(\NS{Z})$.}
We search for a primitive embedding $i\colon \NS{Z}\inj L_{26}$
inducing  $i_{\PPP}\colon  \PP{Z}\inj \PPP(L_{26})$
and a Weyl vector $w_0\in L_{26}$
with the following properties,
and look at the tessellation of the $\Roots(\NS{Z})$-chamber
$\NN{Z}$ by $\prV  (i)$-chambers,
where $\prV (i)$ is defined by~\eqref{eq:VVVS}.
\par
\medskip
(I) Let $R$ denote the orthogonal complement of $\NS{Z}$ in $L_{26}$.
We require that $R$ satisfies~\eqref{eq:Rcond},
so that each  $\prV (i)$-chamber has only a finite number of walls.
We also require that $\eta_R\colon \OG(R)\to \OG(\discf{R})$ is surjective.
By Proposition~\ref{prop:unimoduoverlat},
every isometry $g\in \OGplus(\NS{Z})$ 
extends to an isometry  of $L_{26}$.
Hence the action of $\OGplus(\NS{Z})$
preserves the tessellation of  
 $\PP{Z}$ by $\prV (i)$-chambers.
In particular, 
the action of $\Aut(Z)$ on $\NN{Z}$
preserves the tessellation of  
 $\NN{Z}$ by $\prV (i)$-chambers. 
\par
\medskip
(II)
Let $ D$ be the closed subset $i_{\PPP} \inv(\CCC(w_0))$ of $ \PP{Z}$.
We require that $ D$ contains an ample class in its interior.
Then  $ D$ is a $\prV (i)$-chamber contained in  $\NN{Z}$.
\begin{definition}
The $\prV (i)$-chamber $D$ is called the \emph{initial chamber} of this procedure.
A wall $(v)\sperp\cap D$ of $ D$ is called an \emph{outer-wall}
if $(v)\sperp$ defines a wall of the $\Roots(\NS{Z})$-chamber $\NN{Z}$, 
that is, if there exists a root $r$ of $\NS{Z}$ such that $(v)\sperp=(r)\sperp$.
We call the wall $(v)\sperp\cap D$   an \emph{inner-wall} otherwise.
Let $\WWW_{\out}(D)$ and $\WWW_{\inn}(D)$  denote the set of outer-walls and inner-walls,
respectively.
\end{definition}
We calculate the set of  walls of the initial chamber $ D$.
Since each outer-wall corresponds to a smooth rational curve on $Z$ by Proposition~\ref{prop:NZNZg},
we obtain  a configuration of smooth rational curves on $Z$
from $\WWW_{\out}(D)$.
\par
\medskip
(III)
We calculate $\Aut( D):=\shortset{g\in \OGplus (\NS{Z})}{ D^g= D}$.
By Corollary~\ref{cor:Npreserve},
any element of $\Aut( D)$
preserves  $\NN{Z}$.
Therefore the group
\begin{equation}\label{eq:AutZindD}
\Aut(Z,  D):=\set{g\in \Aut( D)}{\textrm{$g$ satisfies the period condition~\eqref{eq:periodcondition}}}
\end{equation}
is contained in $\Aut(Z)$.
We find  an ample class $h$ in the interior  of $ D$
such that $h^g=h$  for all $g\in \Aut(Z,  D)$.
Then $\Aut(Z,  D)$ is equal to the projective automorphism group
$\Aut(Z, h)$.
\par
\medskip
(IV)
Note that $\Aut(Z,  D)=\Aut(Z, h)$ acts on $\WWW_{\out}(D)$ and $\WWW_{\inn}(D)$.
We decompose $\WWW_{\inn}(D)$ into the orbits under the action 
of  $\Aut(Z,  h)$:
$$
\WWW_{\inn}(D)=O_1\cup \dots \cup O_J.
$$
From each orbit $O_j$,
we choose a wall $(v_j)\sperp \cap D$,
and calculate a Weyl vector $w_j\in L_{26}$
such that $ D_j:=i_{\PPP}\inv (\CCC(w_j))$ is the $\prV (i)$-chamber
adjacent to $ D$ across $(v_j)\sperp\cap D$.
Since $(v_j)\sperp\cap \NN{Z}$ is not a wall of $\NN{Z}$,
the $\prV (i)$-chamber $ D_j$ is contained in $\NN{Z}$.
For each $j=1, \dots, J$, we find  an isometry $g_j$ of $\OGplus (\NS{Z})$
that satisfies the period condition~\eqref{eq:periodcondition} and $ D^{g_j}= D_j$.
Note that each $g_j$ preserves $\NN{Z}$ by  Corollary~\ref{cor:Npreserve},
and hence  $g_j\in \Aut(Z)$.
Note also that, 
for each inner-wall  $(v\sprime)\sperp \cap D\in O_j$,
there exists 
 a conjugate $g\sprime\in \Aut(Z)$ of  $g_j$ by $\Aut(Z, h)$ 
that maps $D$  to the $\prV(i)$-chamber adjacent to $D$ across the wall $(v\sprime)\sperp \cap D$.
\par
\medskip
(V)
Under the assumptions given in (I)-(IV), 
the group $\Aut(Z)$ is 
generated by $\Aut(Z, h)$ and the automorphisms  $g_1, \dots, g_J$.
Moreover, 
the tessellation of $\NN{Z}$ by $\prV(i)$-chambers is $\Aut(Z)$-transitive, and 
the mappings $g\mapsto h^{g}$ and $g\mapsto  D^{g}$ 
give one-to-one correspondences between the following sets:
\begin{itemize}
\item The set of cosets $\Aut(Z, h)\backslash\Aut(Z)$.
\item The set of $\prV (i)$-chambers contained in $\NN{Z}$.
\item The subset $\shortset{h^{g}}{g\in \Aut(Z)}$ of $\NS{Z}$.
\end{itemize}
Moreover, 
considering 
the reflections with respect to the roots $r$ defining the outer-walls $(r)\sperp\cap D$ of $D$,
we see that, under the assumptions given in (I)-(IV),
the tessellation of $\PP{Z}$ by $\prV(i)$-chambers is $\OGplus(\NS{Z})$-transitive.
\par
\medskip
The method  described in this section was applied by Kondo~\cite{KondoKmJac}
to the calculation of  the automorphism group of a generic Jacobian Kummer surface,
and since then, 
many studies have been done on the automorphism groups of various $K3$ surfaces
(see the references of~\cite{ShimadaAlgo}).
This method was also applied to the study of
automorphism group of an Enriques surface in~\cite{ShimadaSch}~and~\cite{SV2019}.
\section{Borcherds' method for $X_0$ and $X_3$}\label{sec:BorcherdsX0X3}
Recall from Section~\ref{subsec:ShiodaKK}  that we use the following notation:
$$
\NS{3}:=\NS{X_3}, \;
\PP{3}:= \PP{X_3},  \;
\NN{3}:=\NN{X_3},  \;
\quad
\NS{0}:=\NS{X_0},  \;
\PP{0}:= \PP{X_0},  \;
\NN{0}:=\NN{X_0}.
$$
\subsection{Borcherds' method for $X_3$}\label{subsec:AutX3} 
%
%______found autreldeg 10 polreldeg 6 [ "A2", "A2", "A2", "A2", "A1", "A1", "A1", "A1", "A1", "A1" ] 1 
%______found autreldeg 31 polreldeg 9 [ "A3", "A3", "A3", "A3", "A1", "A1", "A1", "A1", "A1", "A1" ] 1 
\begin{table}
$$
\renewcommand{\arraystretch}{1.2}
\begin{array}{cccclc}
\textrm{orbit} & \intf{v, v} & \intf{v, h_3} & \intf{h_3, \dpp\sprime_d} & \Sing (\dpp\sprime_d)  &  d=\intf{h_3, h_3^{g(\dpp\sprime_d)}} \\
\hline
O\sprime_{648} & -4/3 & 2 & 6 & 4 A_2+6 A_1  & 10 \\
O\sprime_{5184} & -2/3 & 3  & 9 & 4 A_3+6 A_1   & 31
\end{array}
$$
\caption{Inner-walls of  $D_3$}\label{table:extraautsX3}
\end{table}
We identify $X_3$ and $F_3$ via Shioda's isomorphism explained  
in Section~\ref{subsec:LLL40}.
Hence $\NS{3}$ is the N\'eron-Severi lattice  of $F_3$.
In~\cite{KondoShimadaFQ3},
we have obtained  a generating set of $\Aut(X_3)$ by finding
a primitive embedding
$ i_{3}\colon \NS{3}\inj L_{26}$ inducing $ i_{3, \PPP}\colon \PP{3}\inj \PPP(L_{26})$
and a Weyl vector $w_0\in L_{26}$ that satisfy the requirements 
in~Section~\ref{subsec:geomBorcherds}.
The result is as follows.
See~\cite{thecompdata} or~\cite{KondoShimadaFQ3} for the explicit descriptions of 
$ i_{3}$, $w_0$, and other computational data.
\par
We have $\discg{\NS{3}}\cong (\Z/3\Z)^2$.
The group $\OG(\discf{\NS{3}})$ is a dihedral group of order $8$,
and $\OG(\discf{\NS{3}}, \omega)$ is a cyclic subgroup  of order $4$.
The orthogonal complement $R_3$ of $\NS{3}$ in $L_{26}$ is 
a negative-definite root lattice of type $2A_2$.
The order of $\OG(R_3)$ is $288$,
the order of $\OG(\discf{R_3})$ is $8$,
and the natural homomorphism $\OG(R_3)\to \OG(\discf{R_3})$ is surjective.
We put
$$
 D_3:= i_{3, \PPP}\inv (\CCC(w_0)).
$$
Then $ D_3$   contains the class
$h_3\in \NS{3}$ of a hyperplane section of $X_3=F_3\subset \P^3$ in its interior.
Hence $ D_3$ is a $\prV( i_{3})$-chamber.
The set $\WWW_{\out} ( D_3)$ of outer-walls of the initial chamber  $ D_3$ is equal to 
$\shortset{(\ell)\sperp\cap D_3}{\ell\in \LLL_{112}}$.
Because
$$
h_3=\frac{1}{28}\sum_{\ell\in \LLL_{112}} [\ell],
$$
the group  $\Aut(X_3,  D_3)$
defined by~\eqref{eq:AutZindD} is equal to $\Aut(X_3, h_3)$,
which is  the projective automorphism group 
$\shortset{g\in \PGL_4(k_3)}{g(F_3)=F_3}=\thePGU$
of $F_3\subset\P^3$.
Hence  $\Aut(X_3,  D_3)$ is of order $13063680$.
The class $h_3$ is in fact  the image of $w_0$ 
under the orthogonal projection $L_{26}\tensor\Q\to \NS{3}\tensor\Q$.
Under the action of $\Aut(X_3, h_3)=\thePGU$,
the set $\WWW_{\inn}( D_3)$ of inner-walls of $ D_3$
is decomposed into two orbits $O\sprime_{648}$ and $O\sprime_{5184}$
of size $648$ and $5184$, respectively.
Each inner-wall $(v)\sperp\cap D_3$ in the orbit $O\sprime_s$ is defined by a primitive  vector $v$ of $\NS{3}\dual$
with the properties given in Table~\ref{table:extraautsX3}, and
there exists a double-plane polarization $\dpp\sprime_d \in \NS{3}$
such that the corresponding double-plane involution
$g(\dpp\sprime_d)\in \Aut(X_3)$ maps $ D_3$ to the $\prV (i_3)$-chamber adjacent to $ D_3$
across the wall $(v)\sperp \cap D_3$.
These results prove  the following: 
\begin{theorem}[Kondo--Shimada~\cite{KondoShimadaFQ3}]\label{thm:KondoShimadaFQ3}
The automorphism group  $\Aut(X_3)$  is generated by 
the projective automorphism group $ \Aut(X_3, h_3)= \thePGU$  %of order $13063680$
and two double-plane  involutions $g(\dpp\sprime_{10}), g(\dpp\sprime_{31})$
corresponding the orbits $O\sprime_{648}, O\sprime_{5184}$  
of the action of $\thePGU$ on the set $\WWW_{\inn} (D_3)$ of inner-walls of
the initial chamber $D_3$.
\end{theorem}
\subsection{Borcherds' method for $X_0$}\label{subsec:AutX0} 
\begin{table}
$$
\renewcommand{\arraystretch}{1.2}
\begin{array}{cccclc}
\textrm{orbit} & \intf{v, v} & \intf{v, h_0} & \intf{h_0, \dpp_d} & \Sing (\dpp_d)  &  d=\intf{h_0, h_0^{g(\dpp_d)}} \\
\hline
O_{64} & -5/4 & 5  & 16 & 2 A_3 +3 A_2 +2A_1  & 80 \\
O_{40} & -1 & 6  & 18 & 4 A_3 +3A_1  & 112 \\
O_{160} & -1/2 & 8  & 26 & A_5+2A_4+A_3  & 296 \\
O_{320} & -1/4 & 9  &  38 & 2A_7+A_3+A_1 & 688
\end{array}
$$
\caption{Inner-walls of  $D_0$}\label{table:extraautsX0}
\end{table}
We define an embedding $i_0\colon \NS{0}\inj L_{26}$ by 
\begin{equation}\label{eq:i0i3rho}
 i_{0}:= i_{3}\circ\rho,
\end{equation}
where $ i_{3}\colon \NS{3}\inj L_{26}$ is the embedding used in Section~\ref{subsec:AutX3},
and $\rho\colon \NS{0}\inj \NS{3}$ is the embedding given by the specialization of $X_0$ to $X_3$.
The key observation
of this article is that $ i_{0}$ is equal to the embedding used by Keum--Kondo~\cite{KK2001} for the calculation of $\Aut(X_0)$.
\par
We have $\discg{\NS{0}}\cong (\Z/4\Z)^2$.
The group $\OG(\discf{\NS{0}})$ is  isomorphic to the dihedral group of order $8$,
and the subgroup $\OG(\discf{\NS{0}}, \omega)$ is cyclic of order $4$.
The embedding $ i_{0}$ is primitive and induces
$ i_{0, \PPP}\colon \PP{0}\inj \PPP(L_{26})$.
The orthogonal complement $R_0$ of $\NS{0}$ in $L_{26}$ is a negative-definite root lattice of type $2A_3$.
The order of $\OG(R_0)$ is $4608$,
the order of $\OG(\discf{R_0})$ is $8$,
and the natural homomorphism $ \OG(R_0)\to \OG(\discf{R_0})$ is surjective.
The vector
\begin{equation}\label{eq:h0sum}
h_0:=\frac{1}{2}\sum_{\ell\in \LLL_{40}} [\ell]\;\; \in \;\; \NS{0}\tensor\Q
\end{equation}
is in fact in $\NS{0}$,  and we have $\intf{h_0, h_0}=40$. 
Since $\intf{h_0, \ell}=2$ for all $\ell\in \LLL_{40}$, the class $h_0$ is nef.
Since there exist no roots $r$ of $\NS{0}$ such that $h_0\in (r)\sperp$, 
the class $h_0$ is  ample.
\emph{Let $w_0\in L_{26}$ be the same Weyl vector 
that was used in Section~\ref{subsec:AutX3}.}
The orthogonal projection of $w_0$ to $\NS{0}\tensor\Q$ is equal to $h_0/2$.
(In~\cite{KK2001}, the vector $h_0/2$ is used instead of $h_0$.)
We put
$$
 D_0:= i_{0, \PPP}\inv (\CCC(w_0)).
$$
Then $ D_0$ contains $h_0$ in its interior,
and hence  $ D_0$ is a $\prV ( i_{0})$-chamber.
The set $\WWW_{\out}( D_0)$ of outer-walls of the initial chamber $ D_0$ is equal to 
$\shortset{(\ell)\sperp\cap D_0}{\ell\in \LLL_{40}}$.
We have
\begin{equation}\label{eq:AutX0D0AutX0h0}
\Aut(X_0,  D_0)=\Aut(X_0, h_0), 
\end{equation}
which is of order $3840$ and acts on  $\WWW_{\out}( D_0)$
transitively.
Using the algorithms in Remark~\ref{rem:dppalgo},
we search for double-plane polarizations in $\NS{0}$,
and obtain the following proposition,
which proves Theorem~\ref{thm:KK2001}.
\begin{proposition}\label{prop:X0gdpp}
The action of $\Aut(X_0, h_0)$ 
decomposes 
the set $\WWW_{\inn}(D_0)$ of inner-walls of the initial chamber $ D_0$ 
into four orbits
$O_{64}$, $O_{40}$, $O_{160}$, $O_{320}$,
where $|O_s|=s$.
For each inner-wall $(v)\sperp\cap D_0\in O_s$, 
there exists a double-plane polarization $\dpp_d\in \NS{0}$
such that the corresponding double-plane involution
$g(\dpp_d)\in \Aut(X_0)$ maps $ D_0$ to the $\prV (i_0)$-chamber adjacent to $ D_0$
across the wall $(v)\sperp\cap D_0$.
\qed
\end{proposition}
Each inner-wall $(v)\sperp \cap D_0\in O_s$ 
is defined by a primitive vector $v\in \NS{0}\dual$
with the properties given in Table~\ref{table:extraautsX0}.
See~\cite{thecompdata} 
for the matrix representations of double-plane involutions $g(\dpp_d)$.
\subsection{The group $\Aut(X_0, h_0)$}\label{subsec:AutX0h0}
We investigate the finite group $\Aut(X_0, h_0)$ more closely.
Note that the order $3840$ of this group is the maximum among all
finite subgroups of automorphisms of complex $K3$ surfaces (see Kondo~\cite{KondoMax}).
There exists a natural identification between $\WWW_{\out}(D_0)$ and $\LLL_{40}$.
Therefore, by~\eqref{eq:AutX0D0AutX0h0}, the group $\Aut(X_0, h_0)$ acts on $\LLL_{40}$ faithfully,
and hence $\Aut(X_0, h_0)$ is embedded into  the automorphism group $\Aut(\LLL_{40})$ of the dual graph of $\LLL_{40}$.
On the other hand, since $\gen{\LLL_{40}}=\NS{0}$ (Corollary~\ref{cor:S0LLL40}),
we have an embedding $\Aut(\LLL_{40}) \inj \OGplus(\NS{0})$.
In fact, we confirm by direct calculation the following:
\begin{equation*}
\Aut(X_0, h_0)=\sethd{g\in \Aut(\LLL_{40})}{5.5cm}{\textrm{$g$, as an element of $\OGplus(\NS{0})$,
 satisfies the period condition~\eqref{eq:periodcondition}}}, 
\end{equation*}
and  $\Aut(X_0, h_0)$ is of index $2$ in $\Aut(\LLL_{40})$.
By Propositions~\ref{prop:AutQPG} and~\ref{prop:LLL40QPG1}, we have a natural homomorphism
$\Aut(\LLL_{40})\to \Aut(\PG)$ to the  automorphism group of the Petersen graph $\PG$.
Recall that, in Sections~\ref{subsec:geomQPcovering} and~\ref{subsec:anothermodel},   
we have constructed 
a morphism $\mu_{\C} \colon X_0\to M_{\C}$
that induces the $\QP$-covering map
$\LLL_{40}\to \PG$, and calculated the Galois group $\Gal(\mu)$ in Propositions~\ref{prop:Galois}~and~\ref{prop:Galoct}.
\begin{proposition} \label{prop:GalAutX0h0}
The homomorphism
\begin{equation}\label{eq:fromAutX0h0toAutPG}
\Aut(X_0, h_0)\inj \Aut(\LLL_{40})\to \Aut(\PG)
\end{equation}
is surjective, and its kernel is equal to the Galois group $\Gal(\mu)\cong (\Z/2\Z)^5$.
\end{proposition}
\begin{proof}
By the list of elements of $\Aut(X_0, h_0)$ (see~\cite{thecompdata}),
we see that the homomorphism~\eqref{eq:fromAutX0h0toAutPG} is surjective, 
and its kernel is of order $32$.
Each  generator of $\Gal(\mu)$ given in Propositions~\ref{prop:Galois}~or~\ref{prop:Galoct} preserves $\LLL_{40}$,
and hence   $\Gal(\mu)$ is contained in $\Aut(X_0, h_0)$.
Since  $\mu$  induces the $\QP$-covering map
$\LLL_{40}\to \PG$,  it follows that $\Gal(\mu)$ is contained in the kernel of~\eqref{eq:fromAutX0h0toAutPG}.
Comparing the order, we complete the proof.
\end{proof}
For $v\in \NS{0}$, we put
$$
\Aut(X_0, v):=\set{g\in \Aut(X_0)}{v^g=v}.
$$
Let $f\in \NS{0}$ be the class of a fiber of the Jacobian fibration $\sigma\colon X_0\to \P^1$
defined by~\eqref{eq:theellfib4}.
For each  element $g$ of $\Aut(X_0, f)$, there exists an automorphism  $\bar{g}\in \Aut(\P^1)$
such that the diagram
\begin{equation}\label{eq:gbarg}
\begin{array}{ccc}
X_0 & \xrightarrow{g} & X_0\\
\mapdownleft{\sigma}\;\; & & \mapdownright{\sigma} \\
\P^1 & \xrightarrow{\bar{g}} & \P^1
\end{array}
\end{equation}
commutes,  and hence 
$g$ preserves $\LLL_{40}$.
Therefore
$\Aut(X_0, f)$ is contained in $\Aut(X_0, h_0)$,
and we have a homomorphism
$$
\beta\colon \Aut(X_0, f)\to \Stab (\Cr(\sigma)),
$$
where 
 $\Cr(\sigma):=\{0,\infty, \pm 1, \pm i\}$ is  the set
of critical values of $\sigma$,
and $\Stab (\Cr(\sigma))$ is
the stabilizer subgroup of  $\Cr(\sigma)$ in  $\Aut(\P^1)$.
\par
We  have the inversion  $\iota_{\sigma}\colon X_0\to X_0$
of the Jacobian fibration $\sigma$.
We also have a subgroup $T_{\sigma}$ of $\Aut(X_0, f)$ consisting of  translations by 
the $16$ sections of $\sigma$.
\begin{proposition}\label{prop:exactAutX0f}
The order of $\Aut(X_0, f)$ is $768$.
The image of    $\beta$ is isomorphic to $\SSSS_4$,
and the kernel of $\beta$ is equal to 
the subgroup $T_{\sigma}\semidirectproduct \gen{\iota_{\sigma}}$ of $\Aut(X_0, f)$.
\end{proposition}
\begin{proof}
By means of $\rho_{\LLL}\colon \LLL_{40}\inj \LLL_{112}$ and~\eqref{eq:sigmaF}, 
we can calculate the  quadrangle $F_c$ in $ \LLL_{40}$ 
consisting of the classes of irreducible components of the singular fiber $\sigma\inv(c)$
for each $c\in \Cr(\sigma)$.
Then $f$ is the sum of vectors in one of these $F_c$,
and hence we can calculate $\Aut(X_0, f)$
from the list of elements of $\Aut(X_0, h_0)$.
Looking at the action of $\Aut(X_0, f)$ on the set of the quadrangles $F_c$,
we see that the image of $\beta$
is isomorphic to $\SSSS_4$
generated by  permutations $(0,-1,-i)(\infty,1,i)$ 
and $(0,-i)(\infty, i)(1, -1)$ of $\Cr(\sigma)$.
Therefore the kernel is of order $32$.
Since $T_{\sigma}\semidirectproduct \gen{\iota_{\sigma, z}}$ 
is of order $32$ and contained in the kernel, we complete the proof. 
\end{proof}
\begin{remark}\label{rem:5fs}
Since $|\Aut(X_0, h_0)|/|\Aut(X_0, f)|=5$,
 the orbit of $f$ under the action of  $\Aut(X_0, h_0)$ 
 consists of $5$ elements $f=f\spar{1}, f\spar{2}, \dots, f\spar{5}$.
 We can easily confirm that 
 $$
 \Gal(\mu)=\bigcap_{\nu=1}^{5} \Aut(X_0, f\spar{\nu}).
 $$
%$\Gal(\mu)$ is the intersection of these $5$ subgroups $ \Aut(X_0, f\spar{\nu})$ 
%of  $\Aut(X_0, h_0)$.
The $5$ classes $f\spar{\nu}$
give rise to $5$ elliptic fibrations $\sigma\spar{\nu}\colon X_0\to \P^1$.
These elliptic fibrations correspond to the choices of 
the $\P^1$-fibration $\varphi_M\colon M\to \P^1$
in~\eqref{eq:varphiM}:  for $\nu=1, \dots, 4$, 
 the class $f\spar{\nu}$ is induced by the pencil of 
 lines passing through the triple point $t_\nu$,
and $f\spar{5}$ is  induced by the pencil of conics passing through all the triple points
(see Remark~\ref{rem:sigmaW}).
Let $h_8\spar{\nu}\in S_0$ be the  polarization of degree $8$
constructed from $\sigma\spar{\nu}\colon X_0\to \P^1$
via the recipe of Barth--Hulek 
explained in Section~\ref{subsec:anothermodel}.
Then we have $\Aut(X_0, f\spar{\nu})=\Aut(X_0, h_8\spar{\nu})$.
\end{remark}
\section{Proof of Theorems~\ref{thm:main}~and~\ref{thm:mainR}}\label{sec:proof} 
We use the same notation as in Section~\ref{sec:BorcherdsX0X3}.
The following fact has been  established. 
\begin{proposition}\label{prop:2simple}
{\rm (1)}
The  tessellation of $\NN{3}$ by $\prV ( i_{3})$-chambers is $\Aut(X_3)$-transitive, and 
the  tessellation of $\PP{3}$ by $\prV ( i_{3})$-chambers is $\OGplus(\NS{3})$-transitive.
\par
{\rm (2)}
The  tessellation of $\NN{0}$ by $\prV ( i_{0})$-chambers is $\Aut(X_0)$-transitive, and
the  tessellation of $\PP{0}$ by $\prV ( i_{0})$-chambers is $\OGplus(\NS{0})$-transitive.
\qed
\end{proposition}
From now on,
we consider $\NS{0}$ as a sublattice of $\NS{3}$ via $\rho\colon \NS{0}\inj \NS{3}$
and $\PP{0}$ as a subspace of $\PP{3}$.
For example, we use notation such as $h_0\in \NS{3}$, $D_0\subset \PP{3}$, $\PP{0}\subset \PP{3}$,~\dots.
By the definition~\eqref{eq:i0i3rho} of $i_0$, we have the following:
\begin{proposition}\label{prop:2tessellations}
The tessellation of $\PP{0}$ by $\prV ( i_{0})$-chambers is
obtained as the restriction to  $\PP{0}$ of the  tessellation of $\PP{3}$ by $\prV ( i_{3})$-chambers.
\qed
\end{proposition}
\subsection{Proof of Theorem~\ref{thm:main}}
First, we show that the restriction homomorphism $\tilde{\rho}$ from $\theOGpair$ to 
$ \OGplus(\NS{0})$ maps 
$\theOGpair\cap\Aut(X_3)$ to $\Aut(X_0)$.
By~Theorem~\ref{thm:TorelliAut},  it suffices to show that,
for each $g\in  \theOGpair\cap\Aut(X_3)$,
the restriction $g|_{\NS{0}}\in \OGplus(\NS{0})$  
satisfies the period condition~\eqref{eq:periodcondition}
and preserves $\NN{0}$.
\begin{lemma}\label{lem:period}
If   $g\in \theOGpair$ satisfies the period condition
$\eta_{\NS{3}}(g)\in \OG(\discf{\NS{3}}, \omega)$ for $X_3$,
then $g|_{\NS{0}}\in \OGplus(\NS{0})$ satisfies the period condition
$\eta_{\NS{0}}(g|_{\NS{0}})\in \OG(\discf{\NS{0}}, \omega)$ for $X_0$.
\end{lemma}
\begin{proof}
Let $Q$ denote the orthogonal complement of $\NS{0}$ in $\NS{3}$.
Then $Q$ is an even negative-definite lattice of rank $2$ with discriminant group 
isomorphic to $(\Z/4\Z)^2\times (\Z/3\Z)^2$.
By the classical theory of Gauss, 
such a lattice is unique up to isomorphism,
and the lattice $Q$ is given by 
a Gram matrix
$$
\left(
\begin{array}{cc}
-12 & 0 \\ 0 & -12
\end{array}
\right).
$$
We consider the commutative diagram in Figure~\ref{fig:period}.
The two isomorphisms in the bottom line of this diagram 
are derived  from the isomorphism $\discf{\NS{3}}\cong \discf{Q}_3$ given by Proposition~\ref{prop:LN}
and the isomorphism $\discf{Q}_2\cong -\discf{\NS{0}}$ given by Proposition~\ref{prop:MN}.
It is easy to verify that  $\OG(Q)$ is a dihedral group  of order $8$, and 
the composites $p_3\circ \eta_Q\colon \OG(Q)\to \OG(\discf{Q}_3)$
 and $p_2\circ \eta_Q\colon \OG(Q)\to \OG(\discf{Q}_2)$
 are isomorphisms,
 where $p_2$ and $p_3$ are projections to the $2$-part and $3$-part,
 respectively.
 Using the image of $\eta_Q\colon \OG(Q)\to \OG(q(Q))$ as the graph of an isomorphism between
 $\OG(\discf{Q}_3)$ and $\OG(\discf{Q}_2)$,
 we obtain an isomorphism $\OG(\discf{\NS{3}})\cong \OG(\discf{\NS{0}})$
that is compatible with the homomorphisms from $\theOGpair$.
 Recall that $\OG(\discf{\NS{3}}, \omega)$ and $\OG(\discf{\NS{0}}, \omega)$ are cyclic of order $4$.
 Since the cyclic subgroup of order $4$ is a characteristic subgroup of the dihedral group of order $8$,
the isomorphism
 $\OG(\discf{\NS{3}})\cong \OG(\discf{\NS{0}})$
 maps $\OG(\discf{\NS{3}}, \omega)$ to $\OG(\discf{\NS{0}}, \omega)$.
\end{proof}
\begin{figure}
\newcommand{\xdownarrow}[1]{\left\downarrow\vbox to #1{}\right.\kern-\nulldelimiterspace}
\newcommand{\mapname}[1]{ \rlap{{\scriptsize $#1$}}}
\newcommand{\longetathree}{\smash{\raise -.4cm \hbox{$\xdownarrow{.6cm}$\mapname{\eta_{\NS{3}}}}}}
\newcommand{\longetazero}{\smash{\raise -.4cm \hbox{$\xdownarrow{.6cm}$\mapname{\eta_{\NS{0}}}}}}
\newcommand{\longetaQ}{\smash{\raise -.1cm \hbox{$\xdownarrow{.3cm}$\mapname{\eta_{Q}}}}}
\newcommand{\ptwopthree}{ \raise 8pt \hbox{{\scriptsize $p_3$}}\swarrow\quad\quad\qquad\searrow \raise 8pt \hbox{{\scriptsize $p_2$}}}
$$
\begin{array}{ccc}
\hbox to 3cm {\qquad$\theOGpair  \qquad \rlap{$\inj$}$} \quad&  \OG(Q)\times \OGplus (\NS{0}) & \hbox to 4.5cm {} \\
\raise 8pt  \hbox{\hskip -4pt\rotatebox{270}{$\inj$}} & \mapdownright{\pr_1} & 
\raise 8pt \hbox{\hskip -3.3cm  
\hbox{\hskip -4pt\rotatebox{340}{$\longrightarrow$}} %$\searrow$  
\mapname{\pr_2}}  \\
\OGplus(\NS{3})  & \OG(Q) & \OGplus(\NS{0})\\
\longetathree & \longetaQ &  \longetazero\\
&  \raise -6pt \hbox{$\OG(\discf{Q})$} & \\
 &\raise 10pt \hbox{$\ptwopthree$}&\\
\hbox to 0cm {\hskip -1cm$\OG(\discf{\NS{3}})  \;\;\cong\;\;\  \OG(\discf{Q}_3)$} &\vbox to .2cm {}&\hbox to 0cm {\hskip -3cm $\OG(\discf{Q}_2) \;\;\cong\;\;  \OG(\discf{\NS{0}})$} 
\end{array}
$$
\caption{Commutative diagram for the period condition}\label{fig:period}
\end{figure}
Since we have calculated the embedding $\rho\colon \NS{0}\inj \NS{3}$
in the form of a matrix
and the set $\WWW_{\out}(D_3)\cup \WWW_{\inn}(D_3)$  
of walls  of the initial chamber $ D_3$ for $X_3$ 
 in the form of a list of vectors~(see~\cite{thecompdata}),
we can easily prove the following:
\begin{lemma}\label{lem:facecodim2}
{\rm (1)} The ample class $h_0$ of $X_0$  is contained in $D_3$, and no outer-walls of $D_3$ pass through $h_0$.
In particular, 
$h_0$ belongs to the interior of $\NN{3}$ and hence is ample for $X_3$.
\par
{\rm (2)}  Among the walls $(v)\sperp \cap D_3$ of $D_3$, there exist exactly two walls 
such that the hyperplane $(v)\sperp$ of $\PP{3}$ contains $\PP{0}$.
These two walls $(v_1)\sperp\cap D_3$ and $ (v_2)\sperp\cap D_3$ belong to the orbit $ O\sprime_{648}\subset \WWW_{\inn}(D_3)$.
Moreover, we have 
 $\intf{v_1, v_2}=0$.  \qed
\end{lemma}
Combining Lemma~\ref{lem:facecodim2} with Propositions~~\ref{prop:2simple} and~\ref{prop:2tessellations},
we obtain the following:
\begin{corollary}\label{cor:facecodim2}
{\rm (1)} 
 We have $\PP{0}=(v_1)\sperp\cap (v_2)\sperp$,
where $(v_1)\sperp$ and  $(v_2)\sperp$ are the hyperplanes of $\PP{3}$ given in Lemma~\ref{lem:facecodim2}.
\par
{\rm (2)} 
For each $\prV ( i_{0})$-chamber $D_0\sprime\subset \PP{0}$,
there exist exactly four 
$\prV ( i_{3})$-chambers  
that contain  $D_0\sprime$.
\par
{\rm (3)} 
The initial chamber  $D_0$ for $X_0$  is a face $(v_1)\sperp \cap (v_2)\sperp\cap D_3$
of the initial chamber $ D_3$ for $X_3$,
and the interior of $D_0\subset \PP{0}$ is contained in the interior of $\NN{3}\subset \PP{3}$.
\par
{\rm (4)} 
The four 
$\prV ( i_{3})$-chambers  
containing   $  D_0$
are contained in $\NN{3}$.
In particular, we have  
 $\gamma_1, \gamma_2, \vare\in \Aut(X_3)$
 such that the four 
$\prV ( i_{3})$-chambers  
containing   $  D_0$ are
 $D_3$ and  $ D_3^{\gamma_1}$,  $ D_3^{\gamma_2}$,  $ D_3^{\vare}$.
See Figure~\ref{fig:fourDX3s}.
\qed
\end{corollary}
\begin{figure}
\setlength{\unitlength}{.6truecm}
\begin{picture}(4,4.5)(0,0)
\put(2,2){\line(1,0){2}}
\put(2,2){\line(-1,0){2}}
\put(2,2){\line(0,1){2}}
\put(2,2){\line(0,-1){2}}
\put(0.6,3){\hbox{\small $ D_3$}}
\put(2.8,3){\hbox{\small $ D_3^{\gamma_1}$}}
\put(0.6,.9){\hbox{\small $ D_3^{\gamma_2}$}}
\put(2.8,.9){\hbox{\small $ D_3^{\vare}$}}
\put(1.6, 4.3){\hbox{\tiny $(v_1)\sperp$}}
\put(-1.2, 1.95){\hbox{\tiny $(v_2)\sperp$}}
\end{picture}
\caption{$\prV ( i_{3})$-chambers containing $  D_0$}\label{fig:fourDX3s}
\end{figure}
\begin{remark}
The automorphisms
$\gamma_1$ and $\gamma_2$ of $X_3$ in 
Corollary~\ref{cor:facecodim2}\,(4) can be obtained as conjugates of the
double-plane involution $g(b\sprime_{10})$ by $\thePGU$.
Let $(v\spprime)\sperp \cap D_3$ be the wall of $D_3$ that is mapped to 
the wall $(v_2)\sperp \cap D_3^{\gamma_1}$ of $D_3^{\gamma_1}$ by $\gamma_1$.
Then $(v\spprime)\sperp\cap  D_3$ is an inner-wall belonging to $O\sprime_{648}$,
and hence we have a conjugate $\gamma\spprime$ of $g(b\sprime_{10})$ by $\thePGU$ 
that maps $D_3$ to the $\prV(i_3)$-chamber adjacent to $D_3$ across $(v\spprime)\sperp\cap  D_3$.
Then, as the automorphism $\vare$, we can take $\gamma\spprime\gamma_1$.
See~Section~\ref{subsec:Enriques3} for another construction of $\vare$.
\end{remark}
Let $\pr_3\colon L_{26}\tensor\Q\to \NS{3}\tensor\Q$, $\pr_0\colon L_{26}\tensor\Q\to \NS{0}\tensor\Q$ and 
$\pr_{30}\colon \NS{3}\tensor\Q\to \NS{0}\tensor\Q$ be the orthogonal projections.
Then we have $\pr_{30}\circ \pr_3=\pr_0$.
We put
$$
\prV (\rho):=\set{\pr_{30}(r)}{r\in \Roots(\NS{3}), \;\;\; \intf{\pr_{30}(r), \pr_{30}(r)}<0}.
$$
The restriction to $\PP{0}$
of the tessellation  of $\PP{3}$ by $\Roots(\NS{3})$-chambers is 
the tessellation  of $\PP{0}$ by $\prV (\rho)$-chambers.
The closed subset 
$$
\NN{30}:=\NN{3}\cap \PP{0}
$$
of $\PP{0}$ contains $D_0$ by Corollary~\ref{cor:facecodim2}\,(3), and hence its interior is non-empty.
Therefore  $\NN{30}$ is a $\prV (\rho)$-chamber.
We have  
$$
\Roots(\NS{0})\subset \prV (\rho)\subset \prV ( i_{0}),
$$ 
where the second inclusion follows from  $\Roots(\NS{3})\subset \Roots(L_{26})$
and $\pr_{30}\circ \pr_3=\pr_0$.
It follows  from Remark~\ref{rem:refinement} that  
\begin{equation}\label{eq:DNN}
 D_0\subset \NN{30}\subset \NN{0},
 \end{equation}
and that the $\prV(\rho)$-chamber $\NN{30}$
is tessellated  by $\prV ( i_{0})$-chambers.
If  $g\in  \theOGpair$ preserves $\NN{3}$, 
then  $g|_{\NS{0}}\in \OGplus(\NS{0})$  preserves $\NN{30}$,
and hence preserves $\NN{0}$ by~Corollary~\ref{cor:Npreserve}.
Combining this fact with  Lemma~\ref{lem:period},
we conclude that  every element of the image of $\tilde{\rho}|_{\Aut}$ 
belongs to $\Aut(X_0)$.
\par
\medskip
Next we calculate a generating set of the image of $\tilde{\rho}|_{\Aut}$.
\begin{lemma}\label{lem:facescodim2}
The group $\thePGU=\Aut(X_3, h_3)$ acts transitively on the set of non-ordered pairs
$\{(v)\sperp, (v\sprime)\sperp\}$ of hyperplanes of $\PP{3}$ such that 
$(v)\sperp\cap D_3$ and $ (v\sprime)\sperp\cap D_3$ 
are inner-walls of $D_3$ belonging to $O\sprime_{648}$, and 
such that $\intf{v, v\sprime}=0$.
\end{lemma}
\begin{proof}
As can be seen from the list~\cite{thecompdata} of walls of $D_3$,
for each inner-wall   $(v)\sperp\cap D_3$ in $ O\sprime_{648}$,
the number of inner-walls   $(v\sprime)\sperp\cap D_3$ in $ O\sprime_{648}$
 satisfying $\intf{v, v\sprime}=0$ is $42$.
Comparing $42\times 648 /2=13608$  with Corollary~\ref{cor:13608},
we obtain the proof.
\end{proof}
\begin{corollary}\label{cor:backF}
Let $g$ be an element of $\Aut(X_3)$
such that 
$D_0\sprime:=\PP{0}\cap  D_3^{g}$
is a $\prV ( i_{0})$-chamber, that is, 
$ D_0\sprime$ has an interior point as a subset of 
$\PP{0}$.
Then there exists an element $\gamma\in \thePGU$ such that
$\gamma g\in \Aut(X_3)$ maps the face $ D_0$ of  $ D_3$
to the face $ D_0\sprime$ of  $ D_3^{g}=D_3^{\gamma g}$.
\end{corollary}
\begin{proof}
We put $v_1\sprime:=v_1^{g\inv}$ and $v_2\sprime:=v_2^{g\inv}$,
where $v_1$ and $v_2$ are given in Lemma~\ref{lem:facecodim2}.
Then  $D_0^{\prime g\inv}=\PP{0}^{g\inv}\cap  D_3=(v_1\sprime)\sperp\cap (v_2\sprime)\sperp \cap D_3$ 
is a face of $D_3$, 
which is  the intersection of two perpendicular inner-walls $(v_1\sprime)\sperp\cap D_3$ and 
$(v_2\sprime)\sperp \cap D_3$ 
in $ O\sprime_{648}$.
Hence the existence of $\gamma\in \thePGU$   follows from Lemma~\ref{lem:facescodim2}.
\end{proof}
We put 
\begin{equation}\label{eq:AutX3D0}
\Aut(X_3,  D_0):=\set{g\in \Aut(X_3)}{ D_0^g= D_0}, 
\end{equation}
and compare it with $\Aut(X_0,  D_0)=\Aut(X_0, h_0)$.
Note that $\Aut(X_3,  D_0)$ is a subgroup of $\theOGpair\cap \Aut(X_3)$
containing the kernel of $\tilde{\rho}|_{\Aut}$.
\begin{lemma}\label{lem:inj}
The homomorphism
$\tilde{\rho}|_{\Aut}$ maps 
 $\Aut(X_3,  D_0)$ to $ \Aut(X_0, h_0)$ isomorphically.
In particular, the homomorphism $\tilde{\rho}|_{\Aut}$ is injective,
and the  image of $\tilde{\rho}|_{\Aut}$
contains $\Aut(X_0, h_0)$.
\end{lemma}
\begin{proof}
By Corollary~\ref{cor:facecodim2}~(4), 
the subgroup $\Aut(X_3,  D_0)$ of $\Aut(X_3)$  is contained in the finite subset
\begin{equation}\label{eq:4cosets}
\thePGU\;\;\sqcup\;\; \thePGU\cdot {\gamma_1}\;\;\sqcup\;\; \thePGU \cdot  {\gamma_2}\;\;\sqcup\;\; \thePGU\cdot  {\vare}
\end{equation}
of $\Aut(X_3)$.
For each element $g$ of this subset,
we determine whether $g$ preserves $\PP{0}$ or not.
We see that, in each coset $\thePGU\cdot \gamma$ in~\eqref{eq:4cosets},
 exactly $960$ elements $g$  satisfy $ \PP{0}^{g}= \PP{0}$, and that 
 the set of restrictions $g|_{\NS{0}}$ of these $960\times 4=3840$ elements $g$
 is equal to $\Aut(X_0, h_0)$.
\end{proof}
We discuss the following problem:
Let $(v)\sperp$ be a hyperplane of $\PP{0}$  that 
 defines  a wall of $D_0$.
Determine whether $(v)\sperp$ defines  a wall of 
$\NN{30}$ or not.
\par
Since  $\LLL_{40}\subset \LLL_{112}$,
it immediately follows  that, if $(v)\sperp\cap D_0$ is an outer-wall of $D_0$,  then
$(v)\sperp \cap \NN{30}$  is a wall of $\NN{30}$.
\begin{lemma}\label{lem:wallNX3X0}
Let $(v)\sperp\cap D_0$ be an inner-wall of $ D_0$.
Then $(v)\sperp\cap \NN{30}$ is  a wall of $\NN{30}$ if and only if 
$(v)\sperp \cap D_0\in O_{64}$ or $(v)\sperp   \cap D_0 \in O_{160}$.
\end{lemma}
\begin{proof}
Let $g\in\Aut(X_0)$ be an automorphism that maps $D_0$ to the $\prV(i_0)$-chamber 
adjacent to $D_0$ across the inner-wall $(v)\sperp\cap D_0$
(for example, we can take as $g$  a conjugate  by $\Aut(X_0, h_0)$
of the double plane involution $g(\dpp_d)$ 
corresponding to the orbit $O_s$ containing $(v)\sperp\cap D_0$).
Then  $(v)\sperp\cap \NN{30}$  is a wall of $\NN{30}$  if and only if 
$h_0$ and $h_0^{g}$ , regarded as vectors of $\NS{3}$
via $\rho\colon \NS{0}\inj \NS{3}$,
are \emph{separated by a root in $\NS{3}$}, that is, 
the set
$$
\set{r\in \Roots(\NS{3})}{\textrm{$\intf{h_0, r}$ 
and $\intf{h_0^g, r}$ have different sign}}
$$
is non-empty (see Corollary~\ref{cor:Npreserve}).
We can calculate this set 
using  the algorithm described in Section 3.3 of~\cite{ShimadaChar5}.
\end{proof}
\begin{remark}
The `if\,'-part of Lemma~\ref{lem:wallNX3X0} is refined as follows.
For each positive integer $d$, we put
$$
\CCC_d:=\set{[C]\in \NS{3}}{\textrm{$C$ is a smooth rational curve on $X_3$ such that $\intf{h_3, [C]}=d$}}.
$$
The walls of $\NN{3}$ are in one-to-one correspondence
with the union of these sets $\CCC_d$.
We have $\CCC_1=\LLL_{112}$.
The set $\CCC_d$ can be calculated by induction on $d$.
Indeed,
a root $r$ of $ \NS{3}$ satisfying $\intf{h_3, r}=d$ belongs to $\CCC_d$
if and only if 
there exists no class $r\sprime \in \CCC_{d\sprime}$ with $d\sprime<d$
such that $\intf{r, r\sprime}<0$.
By this method, we obtain the following:
\begin{proposition}
For $d=2,3,5,6$, the set $\CCC_d$ is empty.
We have
$$
|\CCC_1|=112, \quad
|\CCC_4|=18144, \quad
|\CCC_7|=2177280=1632960+544320.
$$
The actions of $\thePGU$ on $\CCC_1$ and on $\CCC_4$ are transitive.
The action of $\thePGU$ on $\CCC_7$ has two orbits
of size $1632960$ and $544320$.
\qed
\end{proposition}
%
%
%1 112 
%2 0 
%3 0 
%4 18144 
%5 0 
%6 0 
%7 2177280 
%
Then we have the following:
\begin{itemize}
\item Among the $64$ walls in $O_{64}$,
$32$ walls are defined by   $(\pr_{30}(r))\sperp$ with $r\in \CCC_1$,
and the other $32$ walls are defined by   $(\pr_{30}(r))\sperp$ with  $r\in \CCC_4$.
\item Among the $160$ walls in $O_{160}$,
$40$ walls are defined by  $(\pr_{30}(r))\sperp$ with  $r\in \CCC_1$,
$80$ walls are defined by  $(\pr_{30}(r))\sperp$ with  $r\in \CCC_4$,
and  $40$ walls are defined by  $(\pr_{30}(r))\sperp$ with  $r\in \CCC_7$.
\end{itemize}
\end{remark}
Note that, if   $g\in \Aut(X_0)$ belongs to the image of $\tilde{\rho}|_{\Aut}$, then 
 $g$ preserves $\NN{30}\subset \NN{0}$.
Hence
the double-plane involutions   $\dppinvol(\dpp_{80})$ and $\dppinvol(\dpp_{296})$ 
corresponding to the orbits $O_{64}$ and $O_{160}$ are \emph{not} in the  image of $\tilde{\rho}|_{\Aut}$.
\begin{lemma}\label{lem:twolifts}
Let $O$ be either $O_{40}$ or $O_{320}$, 
and let $(v)\sperp\cap D_0$ be an element of $O$.
Let $ D_0\sprime$ be the  $\prV ( i_{0})$-chamber
adjacent to $ D_0$ across  $(v)\sperp\cap D_0$.
Then there exists an element $g\sprime$ of $\theOGpair\cap \Aut(X_3)$
such that $g\sprime|_{\NS{0}}$ maps $ D_0$ to $ D_0\sprime$.
\end{lemma}
\begin{proof}
Let $F$ denote the  hyperplane  $(v)\sperp$ of $\PP{0}$
considered as a linear subspace of $\PP{3}$ of codimension $3$.
Let $ D_3\sprime$ be  one of the four $\prV( i_{3})$-chambers such that $D_0\sprime=\PP{0}\cap  D\sprime_3$.
(See~Corollary~\ref{cor:facecodim2}~(2).)
We have
$F\cap D_0=F\cap D_0\sprime=F\cap D_3=F\cap D_3\sprime$,
and this set contains a non-empty open subset of $F$.
Lemma~\ref{lem:wallNX3X0} implies that 
there exists no root $r$ of $\NS{3}$ such that
the hyperplane $(r)\sperp$ of $\PP{3}$ contains 
$F$.
Since $F\cap D_3=F\cap D_3\sprime$,
we see that $D_3$ and $D_3\sprime$ are on the same side of $(r)\sperp$
for any root $r$ of $\NS{3}$, and hence 
$ D_3\sprime$ 
is contained in $\NN{3}$.
Therefore 
we have an element $g\sprime$ of $\Aut(X_3)$
such that $ D_3^{g\sprime}= D_3\sprime$.
By Lemma~\ref{cor:backF},
there exists an element $\gamma$ of $\thePGU$ such that
$\gamma g\sprime$ maps the face $ D_0$ of  $ D_3$
to the face $ D_0\sprime$ of  $ D_3\sprime=D_3^{g\sprime}=D_3^{\gamma g\sprime}$.
Since each of  $ D_0$ and $ D_0\sprime$ contains 
a non-empty open subset of $\PP{0}$,
we see that $\gamma g\sprime\in \Aut(X_3)$ belongs to $\theOGpair$.
Then $\gamma g\sprime|_{\NS{0}}$ maps $D_0$ to $D_0\sprime$.
\end{proof}
Lemmas~\ref{lem:inj}~and~\ref{lem:twolifts} imply that 
  $\dppinvol(\dpp_{112})$ and $\dppinvol(\dpp_{688})$
are in the  image of $\tilde{\rho}|_{\Aut}$.
Let $G$ be the subgroup of $\Aut(X_0)$ generated by $\Aut(X_0, h_0)$ and   
$\dppinvol(\dpp_{112})$ and $\dppinvol(\dpp_{688})$.
Since $G$ is contained in the image of $\tilde{\rho}|_{\Aut}$,
 each $g\in G$ preserves $\NN{30}$.
\begin{lemma}
If a $\prV ( i_{0})$-chamber $D\sprime$ is contained in $\NN{30}$,
then there exists an element $g\in G$ such that $D\sprime=D_0^g$.
\end{lemma}
\begin{proof}
Since $\NN{30}$ is tessellated by $\prV ( i_{0})$-chambers, 
there exists a sequence
$$
D\spar{0}= D_0, \; D\spar{1}, \;  \dots\; , \;  D\spar{N}= D\sprime
$$
of $\prV ( i_{0})$-chambers such that each $D\spar{\nu}$ is contained in $\NN{30}$
and that $D\spar{\nu}$ is adjacent to $D\spar{\nu-1}$ for $\nu=1, \dots, N$.
We prove the existence of $g\in G$ by induction on $N$.
The case $N=0$ is trivial.
Suppose that $N>0$, and let
$g\sprime\in  G$ be an element such that 
$ D_0^{g\sprime}=D\spar{N-1}$.
Note that 
 $g\sprime$ preserves $\NN{30}$.
The $\prV ( i_{0})$-chambers $ D_0$ and $ D^{\sprime g\sprimeinv}$ are adjacent,
and both are contained in $\NN{30}$.
Hence, by Lemma~\ref{lem:wallNX3X0},
the wall of $ D_0$ across which $D^{\sprime g\sprimeinv}$ is adjacent to $ D_0$
is either in $O_{40}$ or in $O_{320}$.
Therefore we have an element 
$g\spprime \in G$ (a conjugate of $\dppinvol(\dpp_{112})$ or $\dppinvol(\dpp_{688})$ by $\Aut(X_0, h_0)$)
such that $D^{\sprime g\sprimeinv}= D_0^{g\spprime}$.
Then $g\spprime g\sprime\in G$
maps $D_0$ to $D\sprime$.
\end{proof}
Let $g$ be an arbitrary element of the image of $\tilde{\rho}|_{\Aut}$.
Since $g$ preserves
$\NN{30}$, 
 there exists an element $g\sprime\in G$ such that $D_0^g=D_0^{g\sprime}$.
 Then $g\sprime g\inv \in \Aut(X_0, h_0)$, and hence $g\in G$.
Thus  the proof of Theorem~\ref{thm:main} is completed.
\qed
\subsection{Proof of Theorem~\ref{thm:mainR}}
By the commutativity of the diagram~\eqref{eq:triangle} and Theorem~\ref{thm:main},
it suffices to prove that the image of 
$\res_0\colon \Aut(\overline{\XXX/R})\to \Aut(X_0)$ contains $\Aut(X_0, h_0)$
and the double-plane involutions $\dppinvol(\dpp_{112})$ and $\dppinvol(\dpp_{688})$.
Let $\pi\colon \XXX\to \Spec R$ be the elliptic modular surface of level $4$
over a discrete valuation ring $R$ of mixed characteristic 
with residue field $k$ of characteristic $3$.
Let $K$ be the fraction field of $R$. 
We put $X_K:=\XXX\tensor_R K$ and $X_{k}:=\XXX\tensor_R k$,
and identify $X_0$ with $X_K\tensor_K\bar{K}$  and $X_3$ with  $X\tensor_k{\bar{k}}$, 
where $\bar{K}$ and $\bar{k}$ are algebraic closures of $K$ and $k$, respectively.
\par
Replacing $R$ by a finite extension of $R$,
we can assume that $h_0$ is the class of a line bundle $L_K$ on $X_K$,
and that every element of $\Aut(X_0, h_0)$ is defined over $K$.
We can extend $L_K$ to a line bundle $\LLL$ on $\XXX$ 
by (21.6.11) of EGA, IV~\cite{EGAIV4}.
Then the class of the line bundle $L_{k}:=\LLL|X_{k}$ on $X_{k}$ 
is $\rho(h_0)\in \NS{3}$.
Hence $L_{k}$ is ample by Lemma~\ref{lem:facecodim2}.
Therefore $\LLL$ is ample relative to $\Spec R$ by  
(4.7.1) of EGA, III~\cite{EGAIII1}. 
We choose $n>0$ such that $\LLL^{\tensor n}$ is very ample relative to $\Spec R$,
embed 
$\XXX$ into  a projective space $\P^N_R$ over $\Spec R$  by $\LLL^{\tensor n}$, 
and regard $\Aut(X_0, h_0)$ 
as the group of projective automorphisms  of $X_K\subset \P^{N}_K$.
Since $X_3$ is not birationally ruled,
we can apply 
the theorem of Matsusaka-Mumford~\cite{MM1964} and conclude that 
every element of $\Aut(X_0, h_0)$ has a lift in $\Aut(\XXX/R)$.
\begin{remark}
The argument in the preceding paragraph is a special case of 
Theorem 2.1 of Lieblich and Maulik~\cite{LM2011}.
\end{remark}
Let $\dpp$ be either $\dpp_{112}$ or $\dpp_{688}$.
Replacing $R$ by a finite extension of $R$,
we can assume that $\dpp$ is the class of a line bundle $M_K$ on $X_K$,
and that each smooth rational curve contracted by $\Phi_{\dpp}\colon X_K\to \P^2_K$
is defined over $K$.
Let $\Sigma(\dpp)\subset \NS{0}$  
be the set of classes of smooth rational curves contracted by $\Phi_{\dpp}$.
We extend $M_K$ to a line bundle $\MMM$ on $\XXX$.
Then the class of the line bundle 
$M_{k}:=\MMM|X_{k}$ on $X_{k}$ is $\rho(b)\in \NS{3}$.
Using the algorithms in  Remark~\ref{rem:dppalgo}, 
we can verify   
that $\rho(\dpp)$ is a double-plane polarization of $X_3$,
and calculate 
the set $\Sigma(\rho(\dpp))\subset \NS{3}$
of classes of smooth rational curves 
contracted by $\Phi_{\rho(\dpp)}\colon X_k\to \P^2_k$.
Then we have the following equality:
\begin{equation}\label{eq:rhoSigma}
\Sigma(\rho(\dpp))=\rho(\Sigma(\dpp)).
\end{equation}
Since the complete linear systems  
$|M_K|$ and $|M_k|$ are of dimension $2$ and fixed-point free,
we see that $\pi_* \MMM$ is  free of rank $3$  over $R$
and defines  a morphism
$$
\wt{\Phi}\colon \XXX \to \P^2_R
$$
over $R$.
We execute,  over $R$,  Horikawa's canonical resolution for double coverings
branched along a curve with only $ADE$-singularities
(see Section 2 of~\cite{Horikawa1975}).
Let $C_{1, K}, \dots, C_{\mu, K}$ be the smooth rational curves on $X_K$ contracted by $\Phi_{\dpp}$,
where $\mu$ is the total Milnor number of the singularities of the branch curve of $\Phi_{b}$ 
(and hence  of $\Phi_{\rho(b)}$).
It follows from~\eqref{eq:rhoSigma} that 
the closure $\CCC_j$ of each $C_{j, K}$ in $\XXX$ is a smooth family of rational curves over $\Spec R$,
that $\wt{\Phi}$ contracts $\CCC_j$ to an $R$-valued point $q_{0j}$ of $\P^2_R$
(that is, a section of the structure morphism $\P^2_R\to \Spec R$),
and that $\wt{\Phi}$ is finite of degree $2$  over the complement of $\{q_{01}, \dots, q_{0\mu}\}$ in $ \P^2_R$.
We put $J_0:=\{1, \dots, \mu\}$, $P_0:= \P^2_R$, and let $\beta_0\colon P_0\to  \P^2_R$ be the identity.
Suppose that we have a morphism $\beta_i\colon  P_i\to  \P^2_R$  over $R$ from a smooth $R$-scheme $P_i$
and a subset $J_i\subset J_0$ such that
\begin{enumerate}[(i)]
\item $\wt{\Phi}$ factors as
$$
\XXX\maprightsp{\alpha_i}  P_i \maprightsp{\beta_i}   \P^2_R,
$$ 
\item
$\alpha_i$ contracts $\CCC_j$  to an $R$-valued point $q_{ij}$ of $P_i$ for each $j\in J_i$, and 
\item
$\alpha_i$ is finite of degree $2$  over the complement of $\shortset{q_{ij}}{j\in J_i}$ in $P_i$.
\end{enumerate}
Suppose that $J_i$ is non-empty.
We choose an index $j_0\in J_i$,
and let $\beta\sprime \colon P_{i+1}\to P_{i}$ be the blow-up at the $R$-valued point $q_{ij_0}$.
Let $\beta_{i+1}\colon P_{i+1}\to \P^2_R$ be the composite of $\beta\sprime$ and $\beta_i$.
Then properties (i)-(iii) are satisfied with $i$ replaced by $i+1$
for some $J_{i+1}\subset J_{i}$ with $J_{i+1}\neq J_{i}$.
Indeed, $\alpha_{i+1}$ induces a finite morphism from at least one of 
the $\CCC_{j}$ with $j\in J_{i}$
to the exceptional divisor of $\beta\sprime$.
Therefore, after finitely many steps, we obtain a finite double covering 
$\XXX\to P$ that factors $\wt{\Phi}$,
where $P$ is obtained from $\P^2_R$ by a finite number of blow-ups at $R$-valued points.
Then the deck-transformation of $\XXX\to P$ gives a lift of the double-plane involution $\dppinvol(\dpp)\in \Aut(X_0)$
to  $\Aut(\XXX/R)$.
\qed
\begin{remark}
The double-plane polarizations $\rho(\dpp_{112}), \rho(\dpp_{688})\in \NS{3}$ 
have the following properties with respect to $h_3$:
\begin{eqnarray*}
&&\intf{h_3, \rho(\dpp_{112})}=9,
\;\;
\intf{h_3, h_3^{\gdpp{\rho(\dpp_{112})}}}=34,
\\
&&\intf{h_3, \rho(\dpp_{688})}=19,
\;\;
\intf{h_3, h_3^{\gdpp{\rho(\dpp_{688})}}}=178.
\end{eqnarray*}
\end{remark}
\section{Enriques surface of type $\IV$}\label{sec:Enriques}
Let $Z$ be a $K3$ surface 
defined over an algebraically closed field of characteristic $\ne 2$.
For an element $g\in \OG^+(\NS{Z})$ of order $2$, we put
$$
\NS{Z}^{+g}:=\set{v\in \NS{Z}}{v^{g}=v},
\quad
\NS{Z}^{-g}:=\set{v\in \NS{Z}}{v^{g}=-v}.
$$
Suppose that $\vare\colon Z\to Z$ is an Enriques  involution,
and let $\pi\colon Z\to Y:=Z/\gen{\vare}$ be the quotient morphism.
Note that the lattice $\NS{Y}$ of numerical equivalence classes of divisors 
on the Enriques surface $Y$ 
is an even unimodular hyperbolic lattice  of rank $10$,
which is unique up to isomorphism.
Then the pull-back homomorphism $\pi^*\colon \NS{Y}\inj \NS{Z}$ induces
an isometry of lattices from $\NS{Y}(2)$ to
$\NS{Z}^{+\vare}$, where $\NS{Y}(2)$ is the lattice obtained from $\NS{Y}$ by multiplying the intersection form by $2$.
Hence the following are satisfied:
 (i) $\NS{Z}^{+\vare}$ is a hyperbolic lattice of rank $10$, 
and (ii) if $M$ is a Gram matrix of $\NS{Z}^{+\vare}$, then $(1/2)M$ is an integer matrix that defines an even unimodular lattice.
Moreover, since $\pi$ is \'etale,
we have that (iii) the orthogonal complement
$\NS{Z}^{-\vare}$
of $\NS{Z}^{+\vare}$ in $\NS{Z}$ contains no roots.
\subsection{Proof of  Proposition~\ref{prop:Enriques}}\label{subsec:Enriques0}
We check conditions (i), (ii), (iii)  for all involutions 
in the finite group $\Aut(X_0, h_0)$.
It turns out that there exist exactly $6$ involutions $\vare\spar{1}, \dots, \vare\spar{6}$ 
satisfying these conditions.
They are conjugate to each other,
and they belong to the subgroup $\Gal(\mu)$ of $\Aut(X_0, h_0)$ (see Proposition~\ref{prop:GalAutX0h0}).
We show that these involutions are Enriques involutions of type $\IV$.
\par
Let $\vare_0$ be one of  $\vare\spar{1}, \dots, \vare\spar{6}$.
Recall that  $\sigma\colon X_0\to \P^1$ is the Jacobian fibration defined by~\eqref{eq:theellfib4},
and let $f\in \NS{0}$ be the class of a fiber of $\sigma$.
Since $\vare_0\in \Gal(\mu)$,  we have
$\vare_0\in \Aut(X_0, f)$ by Remark~\ref{rem:5fs}.
Let $F_c\subset \LLL_{40}$ be the  set of classes of  irreducible components of 
the singular fiber $\sigma\inv (c)$ over $c\in \Cr(\sigma)$.
Looking at the action of $\vare_0$ on these $6$ quadrangles $F_{c}$,  we see that 
the element 
 $\bar{\vare}_0\in \Stab(\Cr(\sigma))$  
 defined by the diagram~\eqref{eq:gbarg}
 is of  order $2$ and fixes exactly $2$ points 
of $\Cr(\sigma)$.
Suppose that $F_{c}$ is  fixed by $\vare_0$.
Then $\vare_0$ acts on $F_{c}$ as 
$\ell_0\leftrightarrow \ell_2$ and $\ell_1\leftrightarrow \ell_3$,
where $\ell_0, \dots, \ell_3$ are labelled as in~\eqref{eq:alpha}.
Therefore $\vare_0$ is fixed-point free,
and $Y_0:=X_0/\gen{\vare_0}$ is an Enriques surface.
\par
The Enriques involution $\vare_0$ acts on $\LLL_{40}$ in such a way that,
for any curve $C\in \LLL_{40}$,
we have $C\cap C^{\vare_0}=\emptyset$.
Hence we obtain a configuration of $20$ smooth rational curves on $Y_0$.
It is easy to check that this configuration is isomorphic to
the configuration  of type $\IV$.
By Theorem 6.1~of~\cite{Martin}, 
we see that $Y_0$ is an Enriques surface of type $\IV$.
\qed
\par
\medskip
Using Proposition~\ref{prop:Galoct},
we can describe  the $6$ Enriques involutions $\vare\spar{\nu}$
in $\Gal(\mu)$ as follows.
\begin{proposition}\label{prop:tauijkenr}
The involution $\tau_J \in \Gal(\mu)$
is an Enriques involution if and only if
$|J|=3$ and $J$ contains $\{1, 5\}$ or $\{2,6\}$ or $\{3, 4\}$.
\qed
\end{proposition}
\subsection{Proof of  Theorem~\ref{thm:main2}}\label{subsec:Enriques3}
Let $\vare_0\in \Aut(X_0, h_0)$ be the image of $\vare_3$ 
under  $\tilde{\rho}|_{\Aut}$,
which is one of $\vare\spar{1}, \dots, \vare\spar{6}$.
Since $\vare_0\in \Aut(X_0, h_0)$,
the involution $\vare_3$  preserves the face  $ D_0=\PP{0}\cap D_3$ of $ D_3$.
Therefore $\vare_3$  belongs to  the finite group $\Aut(X_3, D_0)$ defined by~\eqref{eq:AutX3D0}.
We check all involutions in $\Aut(X_3, D_0)$
and  find $\vare_3$ in the form of a matrix acting on $\NS{3}$.
We have $\intf{h_3, h_3^{\vare_3}}=16$.
Indeed,  the $\prV ( i_{3})$-chamber $ D_3^{\vare_3}$  is the chamber $ D_3^{\vare}$
in Figure~\ref{fig:fourDX3s}.
The  action of  $\vare_3$
on the fibers  of
the Jacobian fibration $\sigma\colon X_3\to \P^1$  defined by~\eqref{eq:theellfib4}
is exactly the same  as 
the action of $\vare_0$ on the fibers of the corresponding fibration of $X_0$.
Hence $\vare_3$ is fixed-point free.
Moreover the configuration on 
$Y_3:=X_3/\gen{\vare_3}$
of $20$ smooth rational curves obtained from $\LLL_{40}\subset \LLL_{112}$
is isomorphic to the configuration of type $\IV$, and hence 
$Y_3$ is of type $\IV$ by Theorem 6.1~of~\cite{Martin}.
The set of pull-backs of the smooth rational curves on $Y_3$ by $\pi_3$  
is $\LLL_{40}\subset \LLL_{112}$.
Hence they are lines on $F_3$.
\qed
%
%\begin{remark}\label{rem:X4422B}
%In the set of involutions of $X_3$ that map $D_3$ to $D_3^{\vare}$,
%there exist exactly $3$ involutions 
%that preserve (i) the subset $\LLL_{40}\subset \LLL_{112}$,
%(ii) the classes of a fiber and the zero section of $\sigma\colon X_3\to \P^1$,
%and (iii) exactly $2$ singular fibers of  $\sigma$
%component-wise,
%whereas exchange  the other $2+2$ singular fibers.
%These $3$ involutions are associated with the 
%double covering $X_3\to X_{[4,4,2,2]}$
%in Remark~\ref{rem:X4422A}.
%\end{remark}
%
%\begin{remark}
%In~\cite{KondoFinite},
%Kondo showed that
%the covering $K3$ surface of the complex Enriques surface of type III  is also isomorphic to $X_0$.
%It seems to be an  interesting problem to find 
%an Enriques involution on 
%the Fermat quartic surface $F_3$ that gives rise to the  Enriques surface of type III in characteristic $3$.
%\end{remark}
%
\begin{remark}
Recently,
we have shown in~\cite{SV2019} that $X_0$ has exactly 
$9$ Enriques involutions modulo conjugation in $\Aut(X_0)$,
and that $4$ of the quotient Enriques surfaces
have finite automorphism groups
(of type I, II, III, IV),
whereas the other $5$ have infinite automorphism groups.
\end{remark}
\bibliographystyle{plain}

\end{document}